\documentclass[amsfonts,12pt]{amsart}

\usepackage{amssymb}
\usepackage{graphicx}
\usepackage{stmaryrd}
\usepackage{color}
\usepackage{pifont}
\usepackage{enumerate}

\DeclareMathOperator*{\esssup}{ess\,sup}
\DeclareMathOperator*{\essinf}{ess\,inf}
\newcommand{\ee}{\varepsilon}

\newcommand{\N}{{\mathbb N}}
\newcommand{\Q}{{\mathbb Q}}
\newcommand{\R}{{\mathbb R}}
\newcommand{\Z}{{\mathbb Z}}

\newcommand{\cH}{{\mathcal H}}

\newcommand{\cK}{\mathcal K}
\newcommand{\cL}{\mathcal L}
\newcommand{\cM}{\mathcal M}

\newtheorem{thm}{Theorem}[section]

\newtheorem{lem}[thm]{Lemma}
\newtheorem{cor}[thm]{Corollary}
\newtheorem{rem}[thm]{Remark}

\newtheorem{ex}[thm]{Example}

\newcommand{\inte}{{\mathrm{int}}\,}
\newcommand{\relint}{{\mathrm{relint}}\,}

\newcommand{\cl}{{\mathrm{cl}}\,}
\newcommand{\conv}{{\mathrm{conv}}\,}

\newcommand{\lin}{{\mathrm{lin}}\,}

\newcommand{\Id}{{\mathrm{Id}}}

\newcommand{\cE}{{\mathcal E}}
\newcommand{\cV}{{\mathcal V}}

\newcommand{\di}{\diamondsuit}

\begin{document}
\hfill\today
\bigskip

\title{Rearrangement and polarization}
\author[Gabriele Bianchi, Richard J. Gardner, Paolo Gronchi, and Markus Kiderlen]
{Gabriele Bianchi, Richard J. Gardner, Paolo Gronchi, and Markus Kiderlen}
\address{Dipartimento di Matematica e Informatica ``U. Dini", Universit\`a di Firenze, Viale Morgagni 67/A, Firenze, Italy I-50134} \email{gabriele.bianchi@unifi.it}
\address{Department of Mathematics, Western Washington University,
Bellingham, WA 98225-9063,USA} \email{Richard.Gardner@wwu.edu}
\address{Dipartimento di Matematica e Informatica ``U. Dini", Universit\`a di Firenze, Piazza Ghiberti 27, Firenze, Italy I-50122} \email{paolo.gronchi@unifi.it}
\address{Department of Mathematical Sciences, University of Aarhus,
Ny Munkegade, DK--8000 Aarhus C, Denmark} \email{kiderlen@math.au.dk}
\thanks{First and third author supported in part by the Gruppo
Nazionale per l'Analisi Matematica, la Probabilit\`a e le loro
Applicazioni (GNAMPA) of the Istituto Nazionale di Alta Matematica (INdAM). Second author supported in part by U.S.~National Science Foundation Grant DMS-1402929.  Fourth author supported by the Centre for Stochastic Geometry and Advanced Bioimaging, funded by a grant from the Villum Foundation.}

\maketitle

\begin{abstract}
The paper has two main goals.  The first is to take a new approach to rearrangements on certain classes of measurable real-valued functions on $\R^n$.  Rearrangements are maps that are monotonic (up to sets of measure zero) and equimeasurable, i.e., they preserve the measure of super-level sets of functions.  All the principal known symmetrization processes for functions, such as Steiner and Schwarz symmetrization, are rearrangements, and these have a multitude of applications in diverse areas of the mathematical sciences.  The second goal is to understand which properties of rearrangements characterize polarization, a special rearrangement that has proved particularly useful in a number of contexts. In order to achieve this, new results are obtained on the structure of measure-preserving maps on convex bodies and of rearrangements generally.
\end{abstract}

\section{Introduction}

The idea of replacing an object by one that retains some of its features but is in some sense more symmetrical has been extremely fruitful.  The object may be a set or a function, for example, and the process is then often called symmetrization or rearrangement, respectively.  Steiner symmetrization, introduced by Jakob Steiner around 1836 in his attempt to prove the isoperimetric inequality, is still today a potent tool for establishing crucial inequalities in geometry; see, for example, \cite{HS, LYZ1, LYZ2, LZ}.  The influence of such inequalities, which often have analytical versions, extends far beyond geometry to other areas such as analysis and PDEs, and even outside mathematics, to economics and finance.  The books \cite{BurZ80}, \cite[Chapter~9]{Gar06}, \cite[Chapter~9]{Gru07}, \cite[Chapter~10]{Sch14}, and survey \cite{Gar02} should serve as gateways to the literature.

The topic received a huge boost in 1951 from the classic text of P\'{o}lya and Szeg\"{o} \cite{PS}.  By this time, many other types of symmetrization had been introduced, with similar applications.  The general idea is to find a symmetrization that preserves one physical quantity, while not increasing (or sometimes not reducing) another.  As well as volume, surface area, and mean width, \cite{PS} considers electrostatic capacity, principal frequency (the first eigenvalue of the Laplacian), and torsional rigidity, thereby extending the scope to mathematical physics.  In fact, much of this work was motivated by conjectures of the mathematician-engineer de Saint-Venant (1856) and physicist Strutt (a.k.a.~Lord Rayleigh) (1877), and subsequent work of Hadamard,  Poincar\'e, and others.  The latter included results on rearrangement of functions, already used to great effect in the 1920s by Faber and Krahn.  It turns out that rearranging a function is a notion so fertile that applications arise in areas too diverse for a single text to cover them all in detail: Classical analysis, calculus of variations, complex analysis, convex geometry, geometric measure theory, Banach spaces, potential theory, PDEs, fluid dynamics, mechanics, and meteorology, for example.  Luckily, a beautiful and quite recent survey by Talenti \cite{Tal} contains a comprehensive bibliography, conveniently divided between the main periods of development, from which \cite{B80, Bur, Burt, CR, Dou, F00, Hay, Hen, Kaw, K06, KP, LL, Str}, together with the recent book \cite{Baer}, illustrate the list of areas just mentioned.

The Steiner rearrangement of a function $f$ with respect to an $(n-1)$-dimensional subspace $H$ in $\R^n$, and its natural generalization, the Schwarz rearrangement of $f$ with respect to a $k$-dimensional subspace $H$, are defined via the corresponding symmetrals of its super-level sets; see, for example, \cite[p.~178]{Gru07}.  For convenience, we shall denote either rearrangement by $S_Hf$.  Symmetrals of sets can be identified with rearrangements of their characteristic functions.  The related notion of polarization is more recent but has already stimulated much interest.  According to Solynin \cite[p.~123]{Sol1}, it was first considered for planar sets by Wolontis \cite{Wol} in 1952, and for functions by Ahlfors \cite[p.~34]{A} and Baernstein and Taylor \cite{BT} in the 1970s.  The term itself is due to Dubinin \cite{Dub87}.  The standard {\em polarization} process, sometimes called {\em two-point symmetrization}, with respect to an {\em oriented} $(n-1)$-dimensional (linear) subspace $H$, takes a function $f:\R^n\to \R$ and replaces it by
\begin{equation}\label{pol}
P_Hf(x)=\begin{cases}
\max\{f(x),f(x^{\dagger})\},& {\text{if $x\in H^+$,}}\\
\min\{f(x),f(x^{\dagger})\},& {\text{if $x\in H^-$}},
\end{cases}
\end{equation}
where $^{\dagger}$ denotes the reflection in $H$ and where $H^+$, $H^-$, are the two closed half-spaces bounded by $H$ and determined by its orientation.  If $A\subset \R^n$, then $P_HA$ is the set satisfying $1_{P_HA}=P_H1_A$, where $1_A$ denotes the characteristic function of $A$.  The process has several useful properties: It is equimeasurable, monotonic, $L^p$-contracting, and reduces the modulus of continuity (see Section~\ref{Properties} for the definitions of these terms and references).

The article \cite{BT} demonstrated that polarizations can be more efficient than rearrangements in establishing inequalities involving integrals, and was followed by a number of papers applying polarization to inequalities in the theory of capacities. For example, Dubinin \cite{Dub93} generalized the result of Wolontis \cite{Wol} by showing that the generalized capacity of a condenser in $\R^n$ does not increase under polarization.  Then, in 2000, a landmark study by Brock and Solynin \cite{BS} gave further significance to polarization by showing that the Steiner or Schwarz rearrangements of a function (or symmetrals of a compact set) with respect to a subspace $H$ can be approximated in $L^p(\R^n)$ (or in the Hausdorff metric, respectively) via successive polarizations with respect to a sequence $(H_k)$ of oriented subspaces.  In \cite{BS}, the sequence $(H_k)$ may depend on the function or set, but this dependence was removed by Van Schaftingen \cite{VS1, VS2}.  Indeed, by \cite[Theorem~1 and Section~4.3]{VS2}, the desired approximation of the Steiner or Schwarz rearrangement $S_Hf$ of a suitable function $f$ may be obtained by taking any sequence $(H_k)$ dense in the set of oriented subspaces $J$ such that $J^+$ contains $H$ in its interior and defining $f_1=f$ and
$$f_{k+1}=(P_{H_{k}}\circ P_{H_{k-1}}\circ\cdots\circ P_{H_1})f_{k}$$
for $k\in \N$; then $f_k\to S_Hf$ as $k\to\infty$.  Moreover, polarization is flexible enough to approximate other processes, such as spherical rearrangement and spherical cap symmetrization; see \cite{VS2}.  We refer to \cite[Sections~1.7 and~2.3]{Baer} for a general introduction to polarization. In addition to the references given there and those provided above, more recent work includes \cite{BF, DekV17, Dub14, HacP}.

In \cite{BGG}, an investigation was initiated into symmetrization processes defined, like Steiner symmetrization, with respect to a subspace in $\R^n$.  Characterizations of Steiner symmetrization and others such as Minkowski symmetrization were proved, in terms of basic properties they possess.  It is natural, then, to undertake a similar study with a view to obtaining characterizations of polarization. Here we consider general maps $T:X\to X$, where $X$ is ${\mathcal{M}}(\R^n)$ (or ${\mathcal{M}}_+(\R^n)$), the space of real-valued (or nonnegative, respectively) measurable functions on $\R^n$, the space ${\mathcal{S}}(\R^n)$ of symmetrizable functions in ${\mathcal{M}}(\R^n)$, or the space ${\mathcal{V}}(\R^n)$ of functions in ${\mathcal{M}}_+(\R^n)$ vanishing at infinity.  (See Sections~\ref{subsec:notations} and~\ref{Properties} for definitions and terminology.)  For any $T:X\to X$ we can consider the induced map $\di_T:{\mathcal{L}}^n\to {\mathcal{L}}^n$, where ${\mathcal{L}}^n$ is the class of measurable sets in $\R^n$ of finite measure: If $A\in {\mathcal{L}}^n$, we let $\di_T A$ be the set of all $x\in \R^n$ with $T1_A(x)=1$.

Our new results begin in Section~\ref{General functions} on equimeasurable maps from a subset $X$ of ${\mathcal{M}}(\R^n)$ to itself, i.e., those that preserve the measure of super-level sets, and rearrangements, equimeasurable maps that are also monotonic.  This second and different usage of the term rearrangement---it is now a transformation on a class of functions---is appropriate, since Steiner, Schwarz, and other special rearrangements all have these two properties; see, for example, \cite[Section~II.2]{Kaw}.  Note, however, that the present paper differs in that monotonic really means essentially monotonic, i.e., up to sets of $\cH^n$-measure zero.  The first main result is Lemma~\ref{may8lem}(iii), which states that a rearrangement $T:{\mathcal{M}}(\R^n)\to{\mathcal{M}}(\R^n)$ essentially acts as the identity on constant functions.  Even the special case $T0=0$, essentially, of this natural result seems not to be obvious.  This is applied to prove that a rearrangement $T:{\mathcal{S}}(\R^n)\to{\mathcal{S}}(\R^n)$ essentially satisfies the weak linearity property \eqref{eqoct61} in Lemma~\ref{lemoct6}, which is in turn a crucial ingredient in the second main result, Theorem~\ref{lemapril30}.  The latter provides the explicit formula \eqref{eqoct62} for $Tf$, where $f\in {\mathcal{S}}(\R^n)$, in terms of the induced map $\di_T$ defined above.  Theorem~\ref{lemapril30} is also used in proving Theorem~\ref{coroct241}, which establishes the fundamental formula (see \eqref{eqMe} below)
$\varphi(Tf)=T(\varphi\circ f)$, for $f\in {\mathcal{S}}(\R^n)$ and right-continuous increasing (i.e., non-decreasing) functions $\varphi:\R\to\R$. The formula \eqref{eqoct61}, and versions of \eqref{eqoct62} and Theorem~\ref{coroct241}, appear elsewhere in the literature (compare, for example, \cite[p.~138]{VSW}, \cite[Equations~(3.1) and~(3.6), p.~1762]{BS}, and \cite[Definition~4 and Proposition~3(d)]{VSW}), so we must stress that our approach is quite different and more logical and general. Earlier works such as \cite{BS} and \cite{VSW} begin with set transformations and use them to define special maps on classes of functions, whereas we start with general maps $T$ on classes of functions, define the induced set transformation $\di_T$, and show that in the main situations of interest, $\di_T$ determines $T$.  In particular, \cite[Equation (3.1), p.~1762]{BS}, like \cite[Definition~4]{VSW}, is a definition, not a result, and \cite[Equation (3.6), p.~1762]{BS} and \cite[Proposition~3(d)]{VSW} are deduced from these definitions.  See the Appendix for a more detailed comparison of the two approaches.

Polarization has another basic property in addition to those listed above, namely, it is defined {\em pointwise}, as is clear from (\ref{pol}).  General pointwise maps $T:X\to X$ with respect to an oriented subspace $H$, defined by (\ref{polF}) below, are the focus of Section~\ref{Pointwise}.  Theorem~\ref{thmm1} gives an explicit formula for maps $T:X\to X$ that are both pointwise with respect to $H$ and equimeasurable, where $X={\mathcal{M}}(\R^n)$, ${\mathcal{M}}_+(\R^n)$, ${\mathcal{S}}(\R^n)$, or $\cV(\R^n)$.  The other main result in this section, Theorem~\ref{thmport}, shows that once $T:X\to X$ has these two properties, the others listed above---monotonicity, $L^p$-contracting for $p\ge 1$, and modulus of continuity reducing---are all equivalent and characterize $T$ as being one of four maps: $\Id$, $\dagger$, $P_H$, or $P_H^{\dagger}=\dagger\circ P_H$, where $\Id$ and $^{\dagger}$ denote the identity map and reflection in $H$, respectively.

The pointwise property is a strong one and for the rest of the paper it is discarded.  In Section~\ref{General functions} we examine general maps $T:X\to X$. Our approach is to gain knowledge first about the induced maps $\di_T:{\mathcal{L}}^n\to {\mathcal{L}}^n$. With this aim, in Section~\ref{General sets}, we study general maps $\di:{\mathcal{E}}\subset{\mathcal{L}}^n\to {\mathcal{L}}^n$ between sets in terms of various properties, defined in Section~\ref{Properties}, but for the most part self-explanatory. The first main result is Theorem~\ref{thmm7}, which gives a formula for maps $\di:{\mathcal{K}}^n_n\to {\mathcal{L}}^n$ that are monotonic, measure preserving, respect $H$-cylinders, and map balls to balls, where ${\mathcal{K}}^n_n$ is the set of convex bodies in $\R^n$ and $H$ is an $(n-1)$-dimensional subspace.  The formula (see (\ref{kseg})) shows that for such maps there is a contraction $\varphi_{\di}:\R\to\R$ such that if $K\in {\mathcal{K}}^n_n$, almost all chords of $K$ orthogonal to $H$ are moved orthogonally to $H$ by a distance determined by $\varphi_{\di}$ and the position of the chord relative to $H$.  With this in hand, Corollary~\ref{corlemExtend} states that $\varphi_{\di}(t)=t$, $-t$, $|t|$, or $-|t|$, if and only if $\di$ essentially (i.e., up to sets of the appropriate measure zero) equals $\mathrm{Id}$, $\dagger$, $\di_{P_H}$, or $\di_{P_H}^\dagger=\di_{P_H^\dagger}$, respectively.  The goal then is to find additional or stronger properties that will force $\varphi_{\di}$ to be one of these four functions.  A first attempt replaces the ball-preserving property by one that for measure-preserving maps is stronger and also enjoyed by polarization, namely, that $\di$ preserves perimeter on convex bodies.  However, Theorem~\ref{thmm9convex} shows that maps
$\di:{\mathcal K}_n^n\rightarrow {\mathcal{L}}^n$ that are monotonic, measure preserving, respect $H$-cylinders, and preserve perimeter on convex bodies, are precisely those for which the contraction $\varphi_{\di}$ satisfies the eikonal equation $|\varphi_{\di}'(t)|=1$ for almost all $t\in \R$.  There are clearly solutions to the latter equation other than $\varphi_{\di}(t)=\pm t$ or $\pm |t|$.  To achieve our goal, we find it necessary to focus on maps $\di:{\mathcal E}\rightarrow {\mathcal{L}}^n$, where ${\mathcal E}$ is either the class ${\mathcal{C}}^n$ of compact sets in $\R^n$ or ${\mathcal{L}}^n$, and replace the property that $\di$ respects $H$-cylinders by the stronger one of invariance on $H$-symmetric unions of two disjoint balls.  (Note that while this condition may seem peculiar, it is much weaker than the natural assumption that $\di$ is invariant on all $H$-symmetric sets.)  Thus in Theorem~\ref{thmm9measurable}, we prove that if ${\mathcal{E}}={\mathcal{C}}^n$ or ${\mathcal{L}}^n$ and
$\di:{\mathcal E}\rightarrow {\mathcal{L}}^n$ is monotonic, measure preserving, perimeter preserving on convex bodies, and invariant on $H$-symmetric unions of two disjoint balls, then $\di$ essentially equals $\mathrm{Id}$, $\dagger$, $\di_{P_H}$, or $\di_{P_H}^\dagger$.

Since the maps $\di:{\mathcal{E}}\subset{\mathcal{L}}^n\to {\mathcal{L}}^n$ we study in Section~\ref{General sets} include but are not necessarily symmetrizations, this part of our paper may be viewed as a widening of the scope of \cite{BGG}.

Armed with Theorem~\ref{thmm9measurable}, we prove in Theorem~\ref{thmm10} that if $T:{\mathcal{S}}(\R^n)\to{\mathcal{S}}(\R^n)$ or $T:\cV(\R^n)\rightarrow \cV(\R^n)$ is a rearrangement, and the induced map $\di_T$ is perimeter preserving on convex bodies and invariant on $H$-symmetric unions of two disjoint balls, then $T$ essentially equals $\Id$, $\dagger$, ${P_H}$, or ${P_H}^{\dagger}$.  For maps $T:X\to X$, where $X={\mathcal{M}}(\R^n)$ or ${\mathcal{M}}_+(\R^n)$, the same properties allow the same conclusion for the restriction of $T$ to ${\mathcal{S}}(\R^n)$ or $\cV(\R^n)$, respectively, though not for the unrestricted map.  This is shown in Theorem~\ref{may9cor}.  Both Theorems~\ref{thmm10} and~\ref{may9cor} depend on the main results from Section~\ref{General functions}.

As was mentioned earlier, polarization is $L^p$-contracting and reduces the modulus of continuity, but since compositions of polarizations with respect to different oriented subspaces retain these two properties, they do not seem so useful in classifying polarization among rearrangements.

Throughout the paper we provide examples showing that our main results are best possible in the sense that none of the assumed properties can be omitted or significantly weakened.

We are grateful to David Preiss for communicating the construction in Remark~\ref{remthmm9convex}.

\section{Preliminaries}\label{subsec:notations}

As usual, $S^{n-1}$ denotes the unit sphere and $o$ the origin in Euclidean $n$-space $\R^n$.  Unless stated otherwise, we assume throughout that $n\ge 2$.   The standard orthonormal basis for $\R^n$ is $\{e_1,\dots,e_n\}$ and the Euclidean norm is denoted by $\|\cdot\|$.  The term {\em ball} in $\R^n$ will always mean a {\em closed} $n$-dimensional ball unless otherwise stated.  The unit ball in $\R^n$ will be denoted by $B^n$ and $B(x,r)$ is the ball with center $x$ and radius $r$. If $x,y\in \R^n$ we write $x\cdot y$ for the inner product and $[x,y]$ for the line segment with endpoints $x$ and $y$. If $x\in \R^n\setminus\{o\}$, then $x^{\perp}$ is the $(n-1)$-dimensional subspace orthogonal to $x$. Throughout the paper, the term {\em subspace} means a linear subspace.

If $A$ is a set,  we denote by $\lin A$, $\conv A$, $\cl A$, $\inte A$, $\relint A$, and $\dim A$ the {\it linear hull}, {\it convex hull}, {\it closure},  {\it interior}, {\it relative interior}, and {\it dimension} (that is, the dimension of the affine hull) of $A$, respectively.  If $H$ is a subspace of $\R^n$, then $A|H$ is the (orthogonal) projection of $A$ on $H$ and $x|H$ is the projection of a vector $x\in \R^n$ on $H$.

If $A$ and $B$ are sets in $\R^n$ and $t\ge 0$, then $tA=\{tx:x\in A\}$ and
$$A+B=\{x+y: x\in A, y\in B\}$$
denotes the {\em Minkowski sum} of $A$ and $B$.

When $H$ is a fixed subspace of $\R^n$, we use $A^{\dagger}$ for the {\em reflection} of $A$ in $H$, i.e., the image of $A$ under the map that takes $x\in \R^n$ to $2(x|H)-x$.  If $A^{\dagger}=A$, we say $A$ is {\em $H$-symmetric}. If $H=\{o\}$, we instead write $-A=(-1)A$ for the reflection of $A$ in the origin and {\it $o$-symmetric} for $\{o\}$-symmetric. A set $A$ is called {\em rotationally symmetric} with respect to $H$ if for $x\in H$, $A\cap (H^{\perp}+x)=r_x(B^n\cap H^{\perp})+x$ for some $r_x\ge 0$.  If $\dim H=n-1$, then a compact convex set is rotationally symmetric with respect to $H$ if and only if it is $H$-symmetric.  The term {\em $H$-symmetric spherical cylinder} will always mean a set of the form $$(B(x,r)\cap H)+s(B^n\cap H^{\perp})=(B(x,r)\cap H)\times s(B^n\cap H^{\perp}),$$
where $r,s>0$.  Of course, $H$-symmetric spherical cylinders are rotationally symmetric with respect to both $H$ and $H^{\perp}$.

The phrase {\em translate orthogonal to $H$} means translate by a vector in $H^{\perp}$.

We write ${\mathcal H}^k$ for $k$-dimensional Hausdorff measure in $\R^n$, where $k\in\{1,\dots, n\}$.  When dealing with relationships between sets in $\R^n$ or functions on $\R^n$, the term {\em essentially} means up to a set of $\cH^n$-measure zero.

The Grassmannian of $k$-dimensional subspaces in $\R^n$ is denoted by ${\mathcal{G}}(n,k)$.

We denote by ${\mathcal C}^n$, ${\mathcal G}^n$, ${\mathcal B}^n$, ${\mathcal M}^n$, and ${\mathcal L}^n$ the class of nonempty compact sets, open sets, bounded Borel sets, ${\mathcal H}^n$-measurable sets, and ${\mathcal H}^n$-measurable sets of finite ${\mathcal H}^n$-measure, respectively, in $\R^n$.  Let ${\mathcal K}^n$ be the class of nonempty compact convex subsets of $\R^n$ and let ${\mathcal K}^n_n$ be the class of {\em convex bodies}, i.e., members of ${\mathcal K}^n$ with interior points.  If $K\in {\mathcal K}^n$, then
$$
h_K(x)=\sup\{x\cdot y: y\in K\},
$$
for $x\in\R^n$, defines the {\it support function} $h_K$ of $K$.  The texts by Gruber \cite{Gru07} and Schneider \cite{Sch14} contain a wealth of useful information about convex sets and related concepts such as the {\em intrinsic volumes} $V_j$, $j\in\{1,\dots, n\}$ (see also \cite[Appendix~A]{Gar06}).  In particular, if $K\in {\mathcal K}^n$ and  $\dim K=n$ then $2V_{n-1}(K)$ is the {\em surface area} of $K$. If $\dim K=k$, then $V_k(K)={\mathcal H}^k(K)$ is the {\em volume} of $K$.  By $\kappa_n$ we denote the volume ${\mathcal H}^n(B^n)$ of the unit ball in $\R^n$.

Let ${\mathcal{M}}(\R^n)$ (or ${\mathcal{M}}_+(\R^n)$) denote the set of real-valued (or nonnegative, respectively) measurable functions on $\R^n$ and let ${\mathcal{S}}(\R^n)$ denote the set of functions $f$ in ${\mathcal{M}}(\R^n)$ such that ${\mathcal{H}}^{n}(\{x:f(x)>t\})<\infty$ for $t>\essinf f$.  By ${\mathcal{V}}(\R^n)$, we denote the set of functions $f$ in ${\mathcal{M}}_+(\R^n)$ such that ${\mathcal{H}}^{n}(\{x:f(x)>t\})<\infty$ for $t>0$.  The four classes of functions satisfy
${\mathcal{V}}(\R^n)\subset {\mathcal{S}}(\R^n)\subset {\mathcal{M}}(\R^n)$ and ${\mathcal{V}}(\R^n)\subset {\mathcal{M}}_+(\R^n)\subset {\mathcal{M}}(\R^n)$.  Members of ${\mathcal{S}}(\R^n)$ have been called {\em symmetrizable} (see, e.g., \cite{BS}) and those of ${\mathcal{V}}(\R^n)$ are often said to {\em vanish at infinity}.  Note that the constant functions are symmetrizable but do not vanish at infinity unless they are essentially zero.

If $T:X\to X$, we shall usually write $Tf$ instead of $T(f)$.  If $T_0,T_1:X\to X$ are maps, we say that $T_0$ is {\em essentially equal} to $T_1$ if for $f\in X$, $T_0f(x)=T_1f(x)$ for ${\mathcal{H}}^{n}$-almost all $x\in \R^n$, where the exceptional set may depend on $f$.

\section{Properties of maps}\label{Properties}

Let $i\in \{1,\dots,n-1\}$, let $H\in {\mathcal{G}}(n,i)$ be fixed, and recall that $K^{\dagger}$ is the reflection of $K$ in $H$.  We consider the following properties of a map $\di:{\mathcal{E}}\subset{\mathcal{L}}^n\to {\mathcal{L}}^n$, where it is assumed (here and throughout the paper) that they hold for all $K,L\in {\mathcal{E}}$, that the class ${\mathcal{E}}$ is appropriate for the property concerned, and that sets of ${\mathcal H}^n$-measure zero are ignored.

\smallskip

1. ({\em Monotonic} or {\em strictly monotonic}) \quad $K\subset L \Rightarrow \di K\subset \di L$ (or $\di K\subset \di L$ and $K\neq L \Rightarrow \di K\neq \di L$, respectively).

2. ({\em Measure preserving}) \quad ${\mathcal H}^n(\di K)={\mathcal H}^n(K)$.

3. ({\em Invariant on $H$-symmetric sets})\quad $K^{\dagger}=K\Rightarrow\di K=K$.

4. ({\em Invariant on $H$-symmetric spherical cylinders})\quad If $K=(B(x,r)\cap H)+s(B^n\cap H^{\perp})$, where $r,s>0$ and $x\in\R^n$, then $\di K=K$.

5. ({\em Maps balls to balls})\quad If $K=B(x,r)$, then $\di K=B(x',r')$.

6. ({\em Respects $H$-cylinders})\quad If $K\subset (B(x,r)\cap H)+H^{\perp}$, then $\di K\subset (B(x,r)\cap H)+H^{\perp}$.

\smallskip

Clearly invariance on $H$-symmetric sets implies invariance on $H$-symmetric spherical cylinders.  If ${\mathcal{E}}\subset{\mathcal{B}}^n$ and $\di$ is monotonic and invariant on $H$-symmetric spherical cylinders, then $\di$ respects $H$-cylinders. The assumption ${\mathcal{E}}\subset{\mathcal{B}}^n$ cannot be omitted here, as the following example for $n=2$ shows. Let $H$ be a one-dimensional subspace in $\R^2$, and let $R_{H^{\perp}}E$ denote the reflection of $E\in \cL^2$ in $H^\perp$. Define $\di :\cL^2\to \cL^2$ by
$$
\di E=
\begin{cases}
E, & \text{if $E$ is essentially bounded}, \\
E\cup R_{H^{\perp}}E, &\text{ otherwise}.
\end{cases}
$$
Then $\di$ is invariant on all essentially bounded sets, and in particular on all spherical cylinders. The mapping $\di$ is also monotonic, but does not respect $H$-cylinders.  In Lemma~\ref{lemmar9}, we show that this conclusion can be drawn if additional conditions are imposed.

We need one further property.

\smallskip

7. ({\em Perimeter preserving on convex bodies}) \quad  For each $K\in {\mathcal{K}}^n_n$, $\di K$ is a set of finite perimeter such that $S(\di K)=S(K)=2V_{n-1}(K)$, where $S$ denotes perimeter (see, for example, \cite[p.~170]{EG}).

\smallskip

Let $X\subset {\mathcal{M}}(\R^n)$.  We consider the following properties of a map $T:X\to X$, where the properties are assumed to hold for all $f,g\in X$:

\smallskip

1. ({\em Equimeasurable})\quad
\begin{equation}\label{equi}
{\mathcal{H}}^{n}(\{x: Tf(x)>t\})={\mathcal{H}}^{n}(\{x: f(x)>t\})
\end{equation}
for $t\in \R$.

2. ({\em Monotonic}) \quad $f\le g$, essentially, implies $Tf\le Tg$, essentially.

3. ({\em $L^p$-contracting}) \quad $\|Tf-Tg\|_p\le \|f-g\|_p$ when $f-g\in L^p(\R^n)$.

4. ({\em Modulus of continuity reducing}) \quad $\omega_d(Tf)\le \omega_d(f)$ for $d>0$, where
\begin{equation}\label{omg}
\omega_d(f)=\esssup_{\|x-y\|\le d}|f(x)-f(y)|
\end{equation}
is the {\em modulus of continuity} of $f\in X$.

\smallskip

The map $T$ is called a {\em rearrangement} if it is equimeasurable and monotonic.

If $X$ contains the characteristic functions of sets in ${\cL}^n$, $T:X\to X$, and $A\in \cL^n$, let $\di_T A=\{x: T1_A(x)=1\}$.
In Lemma~\ref{fromSetToFct}(i) below, it is shown that the induced map $\di_T:\cL^n\to \cL^n$ is well defined when $T$ is equimeasurable and $X={\mathcal{M}}(\R^n)$, ${\mathcal{M}}_+(\R^n)$, ${\mathcal{S}}(\R^n)$, or $\cV(\R^n)$.

Two further properties of $T$ depend on some $H\in {{\mathcal{G}}(n,n-1)}$ which in the second case is oriented and bounds closed half-spaces $H^+$ and $H^-$.

\smallskip

5. ({\em Invariant on $H$-symmetric spherical cylinders})\quad If $K=(B(x,r)\cap H)+s(B^n\cap H^{\perp})$, where $r,s>0$, then $\di_T K$ is well defined and $\di_T K=K$, essentially.

6. ({\em Pointwise with respect to $H$}) \quad \begin{equation}\label{polF}
Tf(x)=\begin{cases}
F^+(f(x),f^{\dagger}(x)),& {\text{if $x\in H^+$,}}\\
F^-(f(x),f^{\dagger}(x)),& {\text{if $x\in H^-$}},
\end{cases}
\end{equation}
where $f^{\dagger}(x)=f(x^{\dagger})$ is the reflection
of $f$ in $H$ and where $F^+, F^-:D^2\to D$ coincide on the diagonal $\{(s,s):s\in D\}$. Here
$$D=\{f(x): x\in \R^n,~f\in X\}.$$
Thus $D=\R$ if $X={\mathcal{M}}(\R^n)$ or ${\mathcal{S}}(\R^n)$ and $D=[0,\infty)$ if $X={\mathcal{M}}_+(\R^n)$ or ${\mathcal{V}}(\R^n)$, and in each case $D^2$ is the common domain of $F^+$ and $F^-$.

\smallskip

The pointwise property is inspired by the pointwise operations defined in \cite{GK}. The functions $F^+$ and $F^-$ are said to be {\em associated with} $T$.  One can consider special associated functions such as the {\em $p$th means} $M_p(s,t)=(|s|^p+|t|^p)^{1/p}$ for $p>0$, $M_\infty(s,t)=\max\{s,t\}$, and $M_{-\infty}(s,t)=\min\{s,t\}$.  For $p<0$, one can define $M_p(s,t)=(|s|^p+|t|^p)^{1/p}$, if $st\neq 0$, and $M_p(s,t)=0$, otherwise.  Then polarization (\ref{pol}) corresponds to taking $F^+=M_{\infty}$ and $F^-=M_{-\infty}$ in (\ref{polF}).

Again taking a cue from \cite{GK}, one might consider the following more general version of the pointwise property:
\begin{equation}\label{polgen}
Tf(x)=\begin{cases}
(f*^+f^{\dagger})(x),& {\text{if $x\in H^+$,}}\\
(f*^-f^{\dagger})(x),& {\text{if $x\in H^-$}},
\end{cases}
\end{equation}
where $*^+$ and $*^-$ are two operations between functions on $\R^n$. To assure that \eqref{polgen} is well defined, one would require that $f*^+f^{\dagger}=f*^-f^{\dagger}$ on $H$. Then (\ref{pol}) corresponds to taking $f*^+g=\max\{f,g\}$ and $f*^-g=\min\{f,g\}$.   However, the apparent restriction in the definition (\ref{polgen}) is an illusion, since $f*^+f^{\dagger}$ and $f*^-f^{\dagger}$ may as well be replaced by arbitrary functions of $f$.

Polarization, defined by (\ref{pol}), has all the properties 1--6. For properties~1 and~4, see \cite[Propositions~1.35 and~1.37]{Baer}, respectively, noting that the general assumption in \cite[p.~1]{Baer} that $f\in{\mathcal{S}}(\R^n)$ is not necessary.  Properties~2 and~5 are clear and~6 was discussed above. Property~3 seems only to have been stated with unnecessary extra assumptions, so we provide a proof in Theorem~\ref{thmport} below.
	
\section{Equimeasurable maps and rearrangements}\label{General functions}

If $f\in {\mathcal{V}}(\R^n)$, then $\essinf f=0$ and it follows that if $T:{\mathcal{V}}(\R^n)\to{\mathcal{V}}(\R^n)$, then $\essinf Tf=\essinf f$. We now examine the situation for the other classes of functions.

\begin{lem}\label{may8lem}
\noindent{\rm{(i)}}  If $T:{\mathcal{S}}(\R^n)\to {\mathcal{S}}(\R^n)$ is equimeasurable, then $\essinf Tf=\essinf f$ for $f\in {\mathcal{S}}(\R^n)$.

\noindent{\rm{(ii)}} If $T:{\mathcal{M}}(\R^n)\to {\mathcal{M}}(\R^n)$ is a rearrangement, then $\essinf Tf\ge \essinf f$ for $f\in {\mathcal{M}}(\R^n)$.   Hence, $T:{\mathcal{S}}(\R^n)\to {\mathcal{S}}(\R^n)$.

\noindent{\rm{(iii)}} In either case, $T:\cV(\R^n)\to \cV(\R^n)$ and $T$ is essentially the identity on constant functions.
\end{lem}

\begin{proof}
(i) Suppose that $\essinf Tf\neq\essinf f$ for some $f\in {\mathcal{S}}(\R^n)$.  Choose $t\in \R$ strictly between $\essinf Tf$ and $\essinf f$ and note that (\ref{equi}) is violated, since one side is finite and the other infinite.

(ii) Throughout the proof, we shall ignore sets of $\cH^n$-measure zero. Let $f\in {\mathcal{M}}(\R^n)$.  Suppose that $\essinf f=a>\essinf Tf$.  Then there is a $t>0$ such that
$$E=\{x : Tf(x)\leq a-t\}$$
has positive $\cH^n$-measure.

Let $f_0(x)=a-\|x\|$ for $x\in \R^n$.  Then $f_0\in {\mathcal{M}}(\R^n)$ and $f_0\le f$, so the monotonicity of $T$ implies that $Tf_0\le Tf$.  Consequently, we may choose $t_0>0$ such that
\begin{equation}\label{finite_value_on_E}
\cH^n(\{x : Tf_0(x)>a-t_0\}\cap E)>0.
\end{equation}
Note that $t_0\ge t$ by the definition of $E$.  Define
$$g(x)=
\begin{cases}
\max\{f_0(x),a-t/2\},& {\text{if $x\in t_0B^n$,}}\\
f_0(x),& {\text{if $x\notin t_0B^n$}}.
\end{cases}
$$
Clearly $g\in {\mathcal{M}}(\R^n)$, $f_0\leq g\leq f$, and
\begin{equation}\label{f_and_g_coincide}
\{x : f_0(x)>a-t_0\}=\{x : g(x)>a-t_0\}=t_0B^n.
\end{equation}
We have
\begin{equation}\label{Tf_equal_Tg}
 \{x : Tf_0(x)>a-t_0\}=\{x : Tg(x)>a-t_0\},
\end{equation}
because the monotonicity of $T$ implies that the set on the right contains the set on the left, and the two sets have the same $\cH^n$-measure, by \eqref{f_and_g_coincide} and the equimeasurability of $T$.

The monotonicity of $T$ and $g\leq f$ imply that $Tg\leq Tf$. In particular, $Tg(x)\leq a-t$ when $x\in E$, so
\begin{equation}\label{outside_E}
 \{x : Tg(x)>a-3t/4\}\cap E=\emptyset.
\end{equation}
Since $t_0\ge t$, we have $\{x : Tg(x)>a-3t/4\}\subset\{x : Tg(x)>a-t_0\}$, so \eqref{Tf_equal_Tg} yields
\begin{equation}\label{inclusion_Tg}
\{x : Tg(x)>a-3t/4\}\subset\{x : Tf_0(x)>a-t_0\}.
\end{equation}
Moreover,
\begin{eqnarray}\label{equal_measure_Tg}
\cH^n(\{x : Tg(x)>a-3t/4\})&=&\cH^n(\{x : g(x)>a-3t/4\})\nonumber\\
&=&\cH^n(t_0B^n)=\cH^n(\{x : f_0(x)>a-t_0\})\nonumber\\
&=&\cH^n(\{x : Tf_0(x)>a-t_0\}).
\end{eqnarray}
Formulas \eqref{inclusion_Tg} and \eqref{equal_measure_Tg} imply that
$$\{x : Tg(x)>a-3t/4\}=\{x : Tf_0(x)>a-t_0\}.$$
But this contradicts \eqref{finite_value_on_E} and \eqref{outside_E} and proves that $\essinf Tf\ge \essinf f$.

Let $f\in {\mathcal{S}}(\R^n)$.  If $t>\essinf Tf$, then $t>\essinf f$, so by \eqref{equi} and the definition of ${\mathcal{S}}(\R^n)$, we have
$$\cH^n(\{x: Tf(x)>t\})=\cH^n(\{x: f(x)>t\})<\infty.$$
Therefore $Tf\in {\mathcal{S}}(\R^n)$, as required.

(iii) If $T:{\mathcal{S}}(\R^n)\to {\mathcal{S}}(\R^n)$ is equimeasurable and $f\in {\cV}(\R^n)$, then $\essinf Tf=\essinf f=0$ by (i), so $Tf\in {\cV}(\R^n)$.  If $T:{\mathcal{M}}(\R^n)\to {\mathcal{M}}(\R^n)$ is a rearrangement and $f\in {\cV}(\R^n)$, then $\essinf Tf\ge \essinf f=0$ by (ii).  This and (\ref{equi}) imply that $Tf\in {\cV}(\R^n)$.  That $T$ is essentially the identity on constant functions is immediate in case (i).  In case (ii), if $f\equiv c$ is constant, we obtain $Tc\ge c$.  This and the equimeasurability of $T$ yield $Tc=c$.
\end{proof}

\begin{ex}\label{may1ex2}
{\rm  Lemma~\ref{may8lem}(ii) and~(iii) do not hold in general if $T: {\mathcal{M}}(\R^n)\to {\mathcal{M}}(\R^n)$ is only assumed to be equimeasurable.
To see this, let
\begin{equation}\label{g2M}
g_1(x)=\left\{\begin{array}{ll}
-1/\|x\|,&x\neq o,\\
0,&x=o.
\end{array}  \right.
\end{equation}
If $A\in \cL^n$, define
$$T_01_A(x)=\begin{cases}
1,& {\text{if $x\in A$,}}\\
g_1(x),& {\text{if $x\not\in A$.}}
\end{cases}$$
Then (\ref{equi}) holds with $f=1_A$ and we may extend $T_0$ arbitrarily to an equimeasurable map from ${\mathcal{M}}(\R^n)$ to itself.  Since $T_01_A\not\in {\mathcal{S}}(\R^n)$, it is not even true that $T_0:{\cV}(\R^n)\to {\mathcal{S}}(\R^n)$.  For the remaining statements in Lemma~\ref{may8lem}(ii) and~(iii), define $T_10=g_1$ and extend the definition of $T_1$ arbitrarily to an equimeasurable map from ${\mathcal{M}}(\R^n)$ to itself. Then $\essinf T_10=-\infty<0=\essinf 0$ and $T_10$ is not a constant function.}
\qed
\end{ex}

If $T: {\mathcal{M}}_+(\R^n)\to {\mathcal{M}}_+(\R^n)$ is equimeasurable, then $T: {\cV}(\R^n)\to {\cV}(\R^n)$ follows immediately from (\ref{equi}).  However, the following example shows that the other statements in Lemma~\ref{may8lem}(ii) and~(iii) do not hold in general if $T: {\mathcal{M}}_+(\R^n)\to {\mathcal{M}}_+(\R^n)$ is a rearrangement.

\begin{ex}\label{aug30ex2}
{\rm  There is a rearrangement $T: {\mathcal{M}}_+(\R^n)\to {\mathcal{M}}_+(\R^n)$ such that $\essinf Tf <\essinf f$ for some $f\in {\mathcal{M}}_+(\R^n)$.  Indeed, define $T$ by letting
$$Tf(x)=\begin{cases}
f(x),& {\text{if $x_1<0$,}}\\
0,& {\text{if $0\le x_1\le 1$,}}\\
f(x-e_1),& {\text{if $x_1>1$,}}
\end{cases}$$
for $f\in {\mathcal{M}}_+(\R^n)$ and $x=(x_1,\dots,x_n)\in\R^n$.  It is easy to check that $T$ is a rearrangement.  If $f\equiv c$ is a constant function and $c>0$, then $\essinf Tf = 0 < c=\essinf f$. It is also not true that $T$ essentially maps constant functions to constant functions.}
\qed
\end{ex}

\begin{ex}\label{aug26ex}
{\rm If $X={\mathcal{M}}(\R^n)$ or ${\mathcal{M}}_+(\R^n)$, there are rearrangements $T:X\to X$ such that $\essinf Tf> \essinf f$ for some $f\in X$.  To see this, call $f\in X$ {\em of type I} if $\cH^n(\{x : f(x)>t\})=\infty$ for $t \ge\essinf f$ and {\em of type II} otherwise, i.e., if there is a $t_0 \ge\essinf f$ such that $\cH^n(\{x : f(x)>t\})<\infty$ for $t > t_0$.  Then define $Tf = f + 1$ if $f$ is of type I and $Tf = f$ if $f$ is of type II.  Clearly, $T: X\to X$ is equimeasurable. If $f\le g$, then either $f$ and $g$ are of the same type, or $f$ is of type II and $g$ is of type I.  It follows that $Tf \le Tg$ and hence that $T$ is a rearrangement.  If $f_0(x)=\|x\|$ for $x\in \R^n$, then $f_0$ is of type I, so $\essinf Tf_0=\essinf(f_0+1)=1>0=\essinf f_0$.}
\qed
\end{ex}

\begin{lem}\label{fromSetToFct}
Let $X={\mathcal{M}}(\R^n)$, ${\mathcal{M}}_+(\R^n)$, ${\mathcal{S}}(\R^n)$, or $\cV(\R^n)$, and let $T:X\to X$ be equimeasurable.

\noindent{\rm{(i)}} The induced map $\di_T:\cL^n\to \cL^n$ given by
\begin{equation}\label{eqIndic}
\di_T A=\{x: T1_A(x)=1\}
\end{equation}
for $A\in\cL^n$ is well defined and measure preserving.

\noindent{\rm{(ii)}} If $X={\mathcal{M}}_+(\R^n)$, ${\mathcal{S}}(\R^n)$, or $\cV(\R^n)$, then $T$ essentially maps characteristic functions of sets in $\cL^n$ to characteristic functions of sets in $\cL^n$, in the sense that for each $A\in \cL^n$,
\begin{equation}\label{eqoct72}
T1_A=1_{\di_TA},
\end{equation}
essentially.
\end{lem}

\begin{proof} (i) If $\alpha>0$ and $A\in \cL^n$, the equimeasurability of $T$ yields
\begin{equation}\label{eqequim1A}
\cH^n(\{x: T(\alpha 1_A)(x)> t\})=\cH^n(\{x: (\alpha 1_A)(x)> t\})=\left\{\begin{array}{ll}
0,&t\ge \alpha,\\
\cH^n(A),&0\le t<\alpha,\\
\infty,&t< 0.
\end{array}  \right.
\end{equation}
Hence $T(\alpha 1_A)(x)\in (-\infty,0]\cup\{\alpha\}$ for $\cH^n$-almost all $x\in \R^n$, and the measurable set
\begin{equation}\label{april301}
A_\alpha=\{x: T(\alpha 1_A)(x)=\alpha\}
\end{equation}
satisfies $\cH^n(A_\alpha)=\cH^n(A)<\infty$. This shows that $A_\alpha\in \cL^n$. Applying this with $\alpha=1$ and setting $\di_T A=A_1$, we obtain \eqref{eqIndic} and the measure-preserving property of $\di_T$.

(ii) It follows from \eqref{eqIndic} that
$$T1_A(x)=1\Leftrightarrow x\in \di_TA\Leftrightarrow 1_{\di_TA}(x)=1.$$
It now suffices to observe that if $X={\mathcal{M}}_+(\R^n)$, ${\mathcal{S}}(\R^n)$, or $\cV(\R^n)$, then $T1_A(x)\in \{0,1\}$ for $\cH^n$-almost all $x\in \R^n$, where for $X={\mathcal{S}}(\R^n)$ we used Lemma~\ref{may8lem}(i).
\end{proof}

\begin{ex}\label{may1ex1}
{\rm  If in Lemma~\ref{fromSetToFct}(i) we have $X={\mathcal{M}}(\R^n)$ and extend the definition (\ref{eqIndic}) to a map $\di_T:\cM^n\to \cM^n$, then it need not be measure preserving.  To see this, let
$g_2:\R^n\to \R$ be defined by
\begin{equation}\label{g1M}
g_2(x)=\left\{\begin{array}{ll}
\frac{\|x\|-1}{\|x\|},&x\neq o,\\
0,&x=o,
\end{array}  \right.
\end{equation}
let $A\in \cM^n$ be such that $\cH^n(A)=\infty$, and let
$$T1_A(x)=\begin{cases}
g_2(x),& {\text{if $x\in A$,}}\\
0,& {\text{if $x\not\in A$.}}
\end{cases}$$
Note that
$$
\cH^n(\{x: T1_A(x)> t\})=\cH^n(\{x: 1_A(x)> t\})=\left\{\begin{array}{ll}
0,&t\ge 1,\\
\infty,&t< 1,
\end{array}  \right.
$$
so (\ref{eqequim1A}) holds with $\alpha=1$. Extend the definition of $T$ to an equimeasurable map from ${\mathcal{M}}(\R^n)$ to ${\mathcal{M}}(\R^n)$ arbitrarily.   Since
$$\di_TA=\{x:T1_A(x)=1\}=\{x\in A:g_2(x)=1\}=\emptyset,$$
$\di_T$ is not measure preserving. \qed}
\end{ex}

The equimeasurable map $T_0$ from Example~\ref{may1ex2} shows that Lemma~\ref{fromSetToFct}(ii) does not hold when $X={\mathcal{M}}(\R^n)$.

\begin{lem}\label{lemoct6}
Let $X={\mathcal{S}}(\R^n)$ or $\cV(\R^n)$ and let $T:X\to X$ be a rearrangement.  For $X={\mathcal{S}}(\R^n)$, $A\in \cL^n$, and $\alpha,\beta\in \R$ with $\alpha\ge 0$, we have
\begin{equation}\label{eqoct61}
T(\alpha 1_A+\beta)=\alpha\, T1_A+\beta,
\end{equation}
essentially.  When $X=\cV(\R^n)$, \eqref{eqoct61} holds, essentially, if $\beta=0$.
\end{lem}

\begin{proof}
The case when $\alpha=0$ follows from Lemma~\ref{may8lem}(iii), so henceforth we assume that $\alpha>0$.

We shall ignore sets of $\cH^n$-measure zero for the remainder of the proof.  We first assume that $\beta=0$.  If $0<\alpha'\le\alpha$, then $\alpha'1_A\le \alpha 1_A$, so $T(\alpha'1_A)\le T(\alpha 1_A)$. Now $T(\alpha' 1_A)(x)=\alpha'$ if and only if $x\in A_{\alpha'}$, where $A_{\alpha'}$ is defined by (\ref{april301}) with $\alpha$ replaced by $\alpha'$, so $T(\alpha 1_A)(x)\ge \alpha'$ for $x\in A_{\alpha'}$.  From the proof of Lemma~\ref{fromSetToFct}(i), we see that $T(\alpha 1_A)(x)=\alpha$ for $x\in  A_{\alpha'}$ and hence $A_{\alpha'}\subset A_{\alpha}$. By \eqref{eqequim1A}, we have $\cH^n(A_{\alpha'})=\cH^n(A)=\cH^n(A_\alpha)$, so $A_{\alpha'}= A_\alpha$.  Consequently, for each $\alpha>0$ we have $A_{\alpha}=A_1$.  Now
$$T(\alpha 1_A)(x)=\alpha\Leftrightarrow x\in A_{\alpha}\Leftrightarrow x\in A_{1}~~\quad{\text{and}}~~\quad
\alpha 1_{\di_T A}(x)=\alpha\Leftrightarrow x\in A_1,$$
so, using \eqref{eqoct72}, we obtain
\begin{equation}\label{eqequimalpha1A}
T(\alpha 1_A)(x)=\alpha 1_{\di_T A}(x)=\alpha\, T1_A,
\end{equation}
as required.  This proves \eqref{eqoct61} when $\beta=0$ and the second statement in the lemma.

Suppose that $\beta\neq 0$ and for convenience let $h=\alpha 1_A+\beta$.  Then $h\in \{\beta,\alpha+\beta\}$. Arguing as in the proof of Lemma~\ref{fromSetToFct}(i), we see that $Th\in(-\infty,\beta]\cup\{\alpha+\beta\}$, and hence, by Lemma~\ref{may8lem}(i), $Th\in \{\beta,\alpha+\beta\}$. It follows that $Th=\alpha1_B+\beta$ for some $\cH^n$-measurable set $B$. For $t\in(\beta,\alpha+\beta)$,
$$\mathcal{H}^n(B)=\mathcal{H}^n\{x: Th(x)>t\}=\mathcal{H}^n\{x: h(x)>t\}=\mathcal{H}^n(A).$$
In view of \eqref{eqoct72}, it will suffice to show that $B=\di_TA$.

Assume that $\beta>0$.  Then $h\ge (\alpha+\beta)1_A$, so using the monotonicity of $T$, and \eqref{eqequimalpha1A} with $\alpha$ replaced by $\alpha+\beta$, we get
$$Th=\alpha 1_B+\beta\ge T\left((\alpha+\beta)1_A\right)=(\alpha+\beta)T1_A=(\alpha+\beta)1_{\di_TA}.$$
Since $\alpha 1_B(x)+\beta=\alpha+\beta\Leftrightarrow x\in B$ and $(\alpha+\beta)1_{\di_TA}(x)=\alpha+\beta\Leftrightarrow x\in \di_TA$, we must have $B=\di_TA$.

Finally, suppose that $\beta<0$.  Let $\gamma=\max\{\alpha+\beta, 1\}$.  Then
$$h=\alpha 1_A+\beta\le (\gamma-\beta)1_A+\beta\le \gamma 1_A,$$
so arguing as above, we find that
\begin{equation}\label{eqoct81}
\alpha 1_B+\beta=Th\le T\left((\gamma-\beta)1_A+\beta\right)
\le (\gamma-\beta)1_C+\beta\le \gamma T1_A=\gamma 1_{\di_TA},
\end{equation}
for some $C$ with $\cH^n(C)=\cH^n(B)=\cH^n(A)$.  Since $(\gamma-\beta)1_C(x)+\beta=\gamma \Leftrightarrow x\in C$ and $\gamma 1_{\di_TA}(x)=\gamma \Leftrightarrow x\in \di_TA$, the right-hand inequality in \eqref{eqoct81} yields $C=\di_TA$.  If $\alpha+\beta \ge 1$, then $\gamma=\alpha+\beta$ and the left-hand inequality in \eqref{eqoct81} similarly yields $B=C$.  If $\alpha+\beta<1$, then $\gamma=1$ and the left-hand inequality in \eqref{eqoct81} becomes $\alpha 1_B\le (1-\beta)1_C$.  Now $\alpha 1_B(x)=\alpha>0 \Leftrightarrow x\in B$ and $(1-\beta)1_C(x)=0 \Leftrightarrow x\not\in C$, so $\cH^n(B\setminus C)=0$.  From this and $\cH^n(B)=\cH^n(C)$, we conclude that $B=C$.  Therefore $B=\di_TA$, as required.
\end{proof}

\begin{thm}\label{lemapril30}
Let $X={\mathcal{M}}(\R^n)$, ${\mathcal{M}}_+(\R^n)$, ${\mathcal{S}}(\R^n)$, or $\cV(\R^n)$ and let $T:X\to X$ be a rearrangement.

\noindent{\rm{(i)}}  The map  $\di_T:\cL^n\to \cL^n$ defined by \eqref{eqIndic} is monotonic.

\noindent{\rm{(ii)}} If $X={\mathcal{S}}(\R^n)$ or $\cV(\R^n)$ and $f\in X$, then
\begin{equation}\label{eqnov11}
\{x: Tf(x)\ge t\}=\di_T\{x: f(x)\ge t\},
\end{equation}
essentially, for $t>\essinf f$.  Moreover, $T$ is essentially determined by $\di_T$, since
\begin{equation}\label{eqoct62}
Tf(x)=\max\left\{\sup\{t\in \Q,\, t>\essinf f: x\in \di_T \{z: f(z)\ge t\}\},\essinf f\right\},
\end{equation}
essentially.
\end{thm}

\begin{proof} (i) If $A\subset B$, then $1_A\le 1_B$ and hence $1=T1_A(x)\le T1_B(x)$ for $\cH^n$-almost all $x\in \di_TA=\{x: T1_A(x)=1\}$. Since $T1_B(x)\in (-\infty,0]\cup\{1\}$ for $\cH^n$-almost all $x\in \R^n$, it is clear that
$\di_TA\subset \{x: T1_B(x)=1\}=\di_TB$, essentially.  Therefore $\di_T$ is monotonic.

(ii) Since $\cV(\R^n)\subset{\mathcal{S}}(\R^n)$, it suffices to consider the case when $X={\mathcal{S}}(\R^n)$.  Let $f\in {\mathcal{S}}(\R^n)$ and let $t>\essinf f$.  If $(t_m)$ is an increasing sequence with $\essinf f<t_m<t$ converging to $t$, the fact that $\cH^n(\{x: f(x)>t_1\})<\infty$ implies that
$$\cH^n(\{x: f(x) \ge t\})=\cH^n\left(\cap_{m=1}^\infty \{x: f(x)>t_m\}\right)=\lim_{m\to{\infty}}
\cH^n\left(\{x: f(x)>t_m\}\right).$$
The same statement holds when $f$ is replaced by $Tf$. The equimeasurability of $T$ then yields
\begin{align}\label{eq:geq}
\cH^n(\{x: f(x) \ge t\})=\cH^n(\{x: Tf(x) \ge t\}).
\end{align}

Assume that $\essinf f>-\infty$.  Let $C=\{x: f(x)\ge t\}$.  It is easy to check that
$$f\ge (t-\essinf f)1_C+\essinf f.$$
By Lemma~\ref{lemoct6} with $\alpha=t-\essinf f$ and $\beta=\essinf f$, and the monotonicity of $T$, we obtain
$$Tf\ge T\left((t-\essinf f)1_C+\essinf f\right)=(t-\essinf f)T1_C+\essinf f,$$
essentially.  This inequality and \eqref{eqIndic} give
$$\di_T\{x:f(x)\ge t\}=\di_TC=\{x:T1_C(x)=1\}\subset \{x: Tf(x)\ge t\},$$
essentially.  The left- and right-hand sides are of equal $\cH^n$-measure by \eqref{eq:geq} and the measure-preserving property of $\di_T$, and are therefore essentially equal.  Hence \eqref{eqnov11} holds when $\essinf f>-\infty$.

Now suppose that $\essinf f=-\infty$.  Let $s<t$ and define $f_s(x)=\max\{f(x),s\}$ for $x\in\R^n$.  Then $\essinf f_s = s$ and
\begin{equation}\label{eqoct71}
\{x:f(x)\ge t\}=\{x:f_s(x)\ge t\}.
\end{equation}
Since $f\le f_s$, the monotonicity of $T$ implies that $\{x:Tf_s(x)\ge t\}$ is essentially contained in $\{x:Tf(x)\ge t\}$. This, \eqref{eqoct71}, and the equimeasurability of $T$ yield $\{x:Tf(x)\ge t\}=\{x:Tf_s(x)\ge t\}$, essentially.  By \eqref{eqnov11} with $f$ replaced by $f_s$, and \eqref{eqoct71} again, we obtain
$$\{x: Tf(x)\ge t\}=\{x: Tf_s(x)\ge t\}=\di_T\{x: f_s(x)\ge t\}=\di_T\{x: f(x)\ge t\},$$
proving that \eqref{eqnov11} holds generally.

Consequently, for each $t>\essinf f$, the symmetric difference
\begin{align}\label{eq:diIsGood}
N_t=\{x: Tf(x)\ge t\}\,\triangle\,\di_T\{x: f(x)\ge t\}
\end{align}
satisfies $\cH^n(N_t)=0$.  According to Lemma~\ref{may8lem}(i), there is a set $N$ such that $\cH^n(N)=0$ and $Tf(x)\ge \essinf f$ when $x\not\in N$.  If $g:\R^n\to \R$, then $g(x)=\sup\{t\in \Q: g(x)\ge t\}$ for $x\in \R^n$. Using this with $g=Tf$ and taking \eqref{eq:diIsGood} into account, we obtain \eqref{eqoct62} for $x\in \R^n\setminus (\cup\{N_t: t\in \Q,\, t>\essinf f\}\cup N)$ and hence for $\cH^n$-almost all $x\in\R^n$.
\end{proof}

Under the assumptions in Lemma~\ref{lemapril30}(ii), it is also true that
\begin{equation}\label{eqnov12}
\{x: Tf(x)> t\}=\di_T\{x: f(x)> t\},
\end{equation}
essentially, for $t>\essinf f$. Indeed, using \eqref{eqnov11}, we have
\begin{eqnarray*}
\{x: Tf(x)> t\}&=&\cup_{n\in \N}\{x: Tf(x)\ge t+1/n\}=
\cup_{n\in \N}\di_T\{x: f(x)\ge t+1/n\}\\
&=&\di_T\cup_{n\in \N}\{x: f(x)\ge t+1/n\}=\di_T\{x: f(x)>t\},
\end{eqnarray*}
essentially, for $t>\essinf f$, where the third equality follows easily from the fact that $\di_T$ is measure preserving and monotonic.  It follows that
\begin{equation}\label{eqnov13}
Tf(x)=\max\left\{\sup\{t\in \Q,\, t>\essinf f: x\in \di_T \{z: f(z)> t\}\},\essinf f\right\},
\end{equation}
essentially, an alternative formula to \eqref{eqoct62}.

We stress that the exceptional set in Theorem~\ref{lemapril30}(ii) cannot be avoided and may depend on $f$. For example, define $T:\cV(\R^n)\rightarrow \cV(\R^n)$ by
$$
Tf(x)=\begin{cases}
0,&\text{if } f(x)<1 ~\text{and}~ x\in \Q^n,\\
f(x),&\text{otherwise}.
\end{cases}
$$
Then $T$ is a rearrangement, but it does not coincide with the identity, although $\di_T=\di_\mathrm{Id}$.  Note also that the supremum over $\Q$ in \eqref{eqoct62} cannot be replaced by the supremum over $\R$, and it is consistent with ZFC that it cannot be replaced by the essential supremum over $\R$; see the remarks after Example~\ref{exoct212}.

With obvious modifications, the following result also applies to rearrangements $T:{\cV}(\R^n)\to {\cV}(\R^n)$.  The equality \eqref{eqMe} also appears in \cite[Proposition~3(d)]{VSW}, where the notation and framework is substantially different, as we explain in the Appendix.  Note that \cite[Proposition~3(d)]{VSW} assumes (in our notation) that $\varphi$ is left-continuous and increasing, but this is not valid in our context, as we show in Example~\ref{exnov14}.

\begin{thm}\label{coroct241}
Let $T:{\mathcal{S}}(\R^n)\to {\mathcal{S}}(\R^n)$ be a rearrangement and let $f\in {\mathcal{S}}(\R^n)$.  If $\varphi:\R\to\R$ is right-continuous and increasing (i.e., non-decreasing), then $\varphi\circ f\in {\mathcal{S}}(\R^n)$ and
\begin{align}\label{eqMe}
\varphi(Tf)=T(\varphi\circ f),
\end{align}
essentially.  It follows that $T(\alpha f+\beta)=\alpha Tf + \beta$, essentially, for $\alpha, \beta\in \R$ with $\alpha\ge 0$.
\end{thm}

\begin{proof}
We do not require the right-continuity of $\varphi$ everywhere, but only at $\essinf f$ when $\essinf f>-\infty$.

We first claim that for any $g\in {\mathcal{S}}(\R^n)$ such that $\varphi$ is right-continuous at $\essinf g$ when $\essinf g>-\infty$, we have
\begin{equation}\label{eqnov41}
\essinf \varphi\circ g=\begin{cases}
\varphi(\essinf g),& {\text{if $\essinf g>-\infty$,}}\\
\inf\varphi,& {\text{if $\essinf g=-\infty$}}.
\end{cases}
\end{equation}
Indeed, if $\essinf g=-\infty$, then $\essinf \varphi\circ g\ge \inf \varphi$ is obvious, while
\begin{equation}\label{eqnov242}\cH^n(\{x:\varphi(g(x))\le \varphi(t)\})\ge \cH^n(\{x:g(x)< t\})>0
\end{equation}
for $t\in \R$ and hence $\inf \varphi=\lim_{t\to-\infty}\varphi(t)\ge \essinf \varphi\circ g$.  If $\essinf g>-\infty$, we have $g\ge \essinf g$, essentially, so $\varphi\circ g\ge \varphi(\essinf g)$, essentially, and hence $\essinf \varphi\circ g\ge \varphi(\essinf g)$.  On the other hand, \eqref{eqnov242} holds for $t>\essinf g$, so for such $t$, $\varphi(t)\ge \essinf \varphi\circ g$.  Then $\varphi(\essinf g)\ge \essinf \varphi\circ g$ follows from the right-continuity of $\varphi$ at $\essinf g$.  This proves \eqref{eqnov41}.

Let $f\in {\mathcal{S}}(\R^n)$.  For $t\in \R$, let $s_t=\inf\{s:\varphi(s)\ge t\}$.  Since $\varphi$ is increasing, we have
\begin{equation}\label{eqnov241}
\{s:\varphi(s)\ge t\}=\begin{cases}
[s_t,\infty),& {\text{if $s_t\in \R$ and $\varphi(s_t)\ge t$,}}\\
(s_t,\infty),& {\text{otherwise}}.
\end{cases}
\end{equation}
As $\varphi$ is also right-continuous at $\essinf f$ when $\essinf f>-\infty$, then under the latter assumption,
\begin{equation}\label{eqlast}
t>\varphi(\essinf f)\Rightarrow s_t>\essinf f.
\end{equation}
Suppose that $\essinf f=-\infty$.  If $t>\inf \varphi$, let $t>t'>\inf\varphi$.  Then $s_{t'}>-\infty$, so \eqref{eqnov241} with $t$ replaced by $t'$ implies that
\begin{equation}\label{eqnov251}
\cH^n(\{x:\varphi(f(x))>t\})\le \cH^n(\{x:\varphi(f(x))\ge t'\})\le\cH^n(\{x:f(x)\ge s_{t'}\})<\infty,
\end{equation}
since this holds trivially when $s_{t'}=\infty$ and in view of $f\in {\mathcal{S}}(\R^n)$ otherwise.  Now suppose that $\essinf f>-\infty$.   If $t>\varphi(\essinf f)$, let $t>t'>\varphi(\essinf f)$.  By \eqref{eqlast} with $t$ replaced by $t'$, we conclude that $s_{t'}>\essinf f$, so \eqref{eqnov251} holds again since $f\in {\mathcal{S}}(\R^n)$.  Since $\essinf \varphi\circ f=\varphi(\essinf f)$, by \eqref{eqnov41} with $g=f$, this proves that $\varphi\circ f\in {\mathcal{S}}(\R^n)$.

We claim that $\essinf \varphi(Tf)=\essinf T(\varphi\circ f)$.  To this end, note that for $g\in {\mathcal{S}}(\R^n)$, we have $\essinf Tg=\essinf g$, by Lemma~\ref{may8lem}(i).  We apply this with $g=f$ and $g=\varphi\circ f$ and \eqref{eqnov41} with $g=Tf$ and $g=f$.   If $\essinf f>-\infty$, we get
$$\essinf \varphi(Tf)=\varphi(\essinf Tf)=\varphi(\essinf f)=\essinf(\varphi\circ f)=\essinf T(\varphi\circ f),$$
while if $\essinf f=-\infty$, then $\essinf Tf=-\infty$ and we obtain
$$\essinf \varphi(Tf)=\inf\varphi=\essinf \varphi\circ f=\essinf T(\varphi\circ f).$$
This proves the claim.

The next step is to prove that
\begin{equation}\label{eqnov243}
\{x:\varphi(Tf(x))\ge t\}=\{x:T(\varphi\circ f)(x)\ge t\},
\end{equation}
essentially, for $t>\essinf \varphi(Tf)=\essinf T(\varphi\circ f)$.  In fact, the latter inequality and \eqref{eqnov41} with $g=Tf$ imply that $s_t>\essinf Tf=\essinf f$, where \eqref{eqlast} was used when $s=\essinf f>-\infty$.  If $\varphi(s_t)\ge t$, we use \eqref{eqnov11} twice to obtain
\begin{align*}
\{x:\varphi(Tf(x))\ge t\}&=\{x:Tf(x)\ge s_t\}=\di_T\{x:f(x)\ge s_t\}
\\&=\di_T\{x:\varphi(f(x))\ge t\}=
\{x:T(\varphi\circ f)(x)\ge t\},
\end{align*}
essentially.  A similar argument, using \eqref{eqnov12} instead of \eqref{eqnov11}, yields \eqref{eqnov243} when $\varphi(s_t)<t$ or $s_t\not\in \R$.

The proof of the first statement in the corollary is concluded by noting that if $g,h\in {\mathcal{S}}(\R^n)$ satisfy $\essinf g=\essinf h$ and $\{x:g(x)\ge t\}=\{x:h(x)\ge t\}$, essentially, for all $t>\essinf g$, then $g=h$, essentially.  Indeed, following the proof of \cite[Lemma~1]{VSW}, we may otherwise assume that there is an $\ee>0$ such that $\cH^n(\{x: h(x)>g(x)+\ee\})>0$.  But
$$\{x: h(x)>g(x)+\ee\}\subset\cup_{n\in\Z,\, n\ee\ge \essinf g}(\{x:h(x)\ge n\ee\}\setminus \{x:g(x)\ge n\ee\}),$$
essentially, and the right-hand side has $\cH^n$-measure zero, a contradiction.

The second statement in the corollary follows immediately from the first on setting $\varphi(t)=\alpha t+\beta$.
\end{proof}

The proof of the previous theorem, as was mentioned at the beginning of it, actually only requires the right-continuity of $\varphi$ at $\essinf f$ when $\essinf f>-\infty$.  The following example shows that this is the weakest possible continuity condition on $\varphi$ for which the theorem holds.

\begin{ex}\label{exnov14}
{\rm  If $\varphi:\R\to \R$ is left-continuous and increasing, it is possible that $f\in {\mathcal{S}}(\R^n)$ but $\varphi\circ f\not\in {\mathcal{S}}(\R^n)$.  Indeed, taking $n=1$ for simplicity, let $\varphi(t)=t$ for $t>0$ and $\varphi(t)=-1$ for $t\le 0$.  Let $f\in {\mathcal{S}}(\R)$ be any function such that $f(x)> 0$ for $x>0$ and $f(x)=0$ for $x\le 0$.  Then $\varphi(f(x))\ge 0$ for $x>0$ and $\varphi(f(x))=-1$ for $x\le 0$, so $\varphi\circ f\not\in {\mathcal{S}}(\R^n)$.  Note that $\varphi$ is continuous everywhere except at $\essinf f=0$, where it is only left-continuous. \qed}
\end{ex}

Equimeasurable maps satisfying \eqref{eqMe} for all right-continuous, increasing $\varphi$ must actually be rearrangements, as we now show.  The first part of the proof uses ideas of Van Schaftingen \cite[Proposition~2.4.1]{VSPhD}.

\begin{lem}\label{lemnov1}
Let $T:{\mathcal{S}}(\R^n)\to {\mathcal{S}}(\R^n)$ be equimeasurable.  Suppose that $\varphi(Tf)=T(\varphi\circ f)$, essentially, whenever $f\in {\mathcal{S}}(\R^n)$ and $\varphi:\R\to\R$ is right-continuous and increasing.  Then $T$ is monotonic and hence a rearrangement.
\end{lem}

\begin{proof}
We shall ignore sets of $\cH^n$-measure zero in this proof. Let $f\in {\mathcal{S}}(\R^n)$ and for $c\in \R$, define $\varphi_c(t)=1$ if $t\ge c$ and $\varphi_c(t)=0$ if $t<c$. Note that $\varphi_c$ is right-continuous and increasing, and for $g\in {\mathcal{S}}(\R^n)$, we have $1_{\{x: g(x)\ge c\}}=\varphi_c\circ g\in {\mathcal S}(\R^n)$, by the first part of the proof of Theorem~\ref{coroct241} (which does not involve $T$), with $\varphi$ and $f$ replaced by $\varphi_c$ and $g$, respectively.  Using this with $g=Tf$ and $g=f$ and our assumption on $T$, we obtain
$$1_{\{x: Tf(x)\ge c\}}=\varphi_c(Tf)=T(\varphi_c\circ f)=T1_{\{x: f(x)\ge c\}}.$$
Hence,
\begin{align*}
\{x: Tf(x)\ge c\}=
\begin{cases}
\di_T{\{x: f(x)\ge c\}},&\text{if $c>\essinf f$},\\
\R^n,& {\text{otherwise}},
\end{cases}
\end{align*}
where \eqref{eqoct72} was used in the first case and Lemma \ref{may8lem}(i) in the second case.  As $Tf(x)=\sup\{c\in \Q: Tf(x)\ge c\}$, the map $T$ satisfies \eqref{eqoct62}.

Suppose that $A\subset B\subset\R^n$ and let $h=1_A+1_B$.  It follows easily from \eqref{eqoct62} that $Th=1_{(\di_T A)\cup \di_T B}+1_{\di_T A}$.  Since $T$ is equimeasurable, we have $\cH^n((\di_T A)\cup \di_T B)=\cH^n(\di_T B)$, so $\di_T A\subset \di_T B$.  Thus $\di_T$ is monotonic.  This implies that if $f, g\in {\mathcal{S}}(\R^n)$ and $f\le g$, then $\di_T\{z:f(z)\ge t\}\subset \di_T\{z:g(z)\ge t\}$, and then $Tf\le Tg$ is a consequence of \eqref{eqoct62}.
\end{proof}

\begin{ex}\label{may1ex3}
{\rm  Let $X={\mathcal{M}}(\R^n)$ or ${\mathcal{M}}_+(\R^n)$.  There is a rearrangement $T:X\to X$ such that $T\neq \Id$ but $T=\Id$ on $\cV(\R^n)$.  In particular, Theorem~\ref{lemapril30}(ii) does not hold.   Indeed, for $f\in X$, let
$$t_f=\inf\{t\ge 0: \cH^n(\{x:f(x) >t\})<\infty\}$$
and let $A_f=\{x:f(x) \ge t_f\}$.  Define
$$Tf(x)=\begin{cases}
f(x),& {\text{if $x\in A_f$,}}\\
\min\{f(x)+1,t_f\},& {\text{if $x\not\in A_f$.}}
\end{cases}$$
Note that if $f\in \cV(\R^n)$, then $t_f=0$ and $A_f=\R^n$, so $T=\Id$ on $\cV(\R^n)$.  Let
$$f(x)=\begin{cases}
0,& {\text{if $x\in B^n$,}}\\
\frac{\|x\|}{\|x\|-1},& {\text{if $x\not\in B^n$.}}
\end{cases}$$
Then $f\in{\mathcal{M}}_+(\R^n)$, $t_f=1$, and $A_f=\R^n\setminus B^n$, so $Tf=f+1_{B^n}\neq f$.

We claim that $T$ is a rearrangement.  Note first that $Tf\ge f$.  Let $f\in X$ and suppose that $t\ge t_f$.  If $f(x)>t$, then $x\in A_f$, so $Tf(x)=f(x)>t$.  Conversely, if $Tf(x)>t$, then $Tf(x)>t_f$, so $x\in A_f$, implying that $Tf(x)=f(x)$ and thus $f(x)>t$.  Hence
$$\{x:Tf(x)>t\}=\{x:f(x)>t\}.$$
Now suppose that $t<t_f$.  Then $\cH^n(\{x:f(x)>t\})=\infty$ by the definition of $t_f$.  But if $f(x)>t$, then $Tf(x)\ge f(x)>t$, so $\cH^n(\{x:Tf(x)>t\})=\infty$.  This proves that $T$ is equimeasurable.

Let $f,g\in X$ satisfy $f\le g$.  Then $t_f\le t_g$.  If $x\in A_f$, then $Tf(x)=f(x)\le g(x)\le Tg(x)$.  If $x\in A_g\setminus A_f$, then
$$Tg(x)=g(x)\ge t_g\ge \min\{f(x)+1,t_g\}\ge \min\{f(x)+1,t_f\}=Tf(x).$$
Finally, if $x\not\in A_f\cup A_g$, then
$$Tg(x)= \min\{g(x)+1,t_g\}\ge \min\{f(x)+1,t_f\}=Tf(x).$$
This proves that $T$ is monotonic. \qed}
\end{ex}

\section{Pointwise maps between functions}\label{Pointwise}

\begin{thm}\label{thmm1}
Let $H\in {{\mathcal{G}}(n,n-1)}$ be oriented, let $X={\mathcal{M}}(\R^n)$, ${\mathcal{M}}_+(\R^n)$, ${\mathcal{S}}(\R^n)$, or $\cV(\R^n)$, and suppose that $T:X\to X$ is pointwise with respect to $H$.	Then $T$ is equimeasurable if and only if its associated functions $F^+$ and $F^-$ satisfy
\begin{equation}\label{Fvalues}
\{ F^+(r,s),F^-(s,r)\}=\{r,s\}
\end{equation}
for $(r,s)\in D^2$, the common domain of $F^+$ and $F^-$.
\end{thm}

\begin{proof}
We first consider the case when $X={\mathcal{M}}(\R^n)$ or ${\mathcal{S}}(\R^n)$, so that $D=\R$.  Assume that $T$ is equimeasurable.  We claim that \begin{equation}\label{eq:diagonId}
F^+(r,r)=F^-(r,r)=r
\end{equation}
for $r\in \R$.  To see this, let $f\equiv r$ be constant on $\R^n$. From \eqref{polF} and the fact that $F^+$ and $F^-$ coincide on the diagonal of $\R^2$, we see that $Tf\equiv F^+(r,r)=F^-(r,r)$ is also constant.  This and (\ref{equi}) yield (\ref{eq:diagonId}) (when $X={\mathcal{S}}(\R^n)$, this is also a consequence of Lemma~\ref{may8lem}(iii)).

Now fix $r,s\in \R$. Let $A\subset \inte H^+$ be compact with ${\mathcal{H}}^{n}(A)>0$, let $c<\min\{r,s\}$, and let
$$f(x)=r 1_{A}(x)+s1_{A^\dagger}(x)+c 1_{\R^n\setminus (A\cup A^\dagger)}(x).$$
Note that $f\in {\mathcal{S}}(\R^n)$.  From \eqref{polF} and \eqref{eq:diagonId}, we have
\begin{equation}\label{polFspec}
Tf(x)=\begin{cases}
F^+(r,s),& {\text{if $x\in A$}},\\
F^-(s,r),& {\text{if $x\in A^\dagger$}},\\
c,& {\text{otherwise}}.
\end{cases}
\end{equation}
If $c<t<\min\{r,s\}$, both sides of \eqref{equi} equal $2{\mathcal{H}}^{n}(A)$, while if $t=\max\{r,s\}$, both sides are zero. Thus
\begin{equation}\label{minmax}
\min\{r,s\}\le  F^+(r,s),F^-(s,r)\le \max\{r,s\}.
\end{equation}
If $r\ne s$, we can choose $t$ in \eqref{equi} with $\min\{r,s\}< t< \max\{r,s\}$.  Then both sides of \eqref{equi} equal ${\mathcal{H}}^{n}(A)$, so $F^+(r,s)\le t$ and $F^-(s,r)>t$ or vice versa. As $\min\{r,s\}<t<\max\{r,s\}$ was arbitrary,
$F^+(r,s)=\min\{r,s\}$ and $F^-(s,r)=\max\{r,s\}$ or vice versa. This proves \eqref{Fvalues} when $r\ne s$, and \eqref{Fvalues} holds trivially when $r=s$ due to \eqref{minmax}.

Now assume that \eqref{Fvalues} holds and let $f\in X$. Define
$$M=\{x\in \R^n: f(x)=f^\dagger(x)\},$$
$$M^\pm=\{x\in H^\pm\setminus M: F^\pm(f(x),f^\dagger(x))=f(x)\},$$
and
$$M^\pm_\dagger=\{x\in H^\pm\setminus M: F^\pm(f(x),f^\dagger(x))=f^\dagger(x)\}.$$
By \eqref{Fvalues}, these five sets form a partition of $\R^n$. Note that if $x\not\in M$, then $Tf(x)=f(x)$ if and only if $x\in M^+\cup M^-$ and $Tf(x)=f^{\dagger}(x)$ if and only if $x\in M^+_{\dagger}\cup M^-_{\dagger}$.  This and the fact that by \eqref{Fvalues} we have $Tf=f$ on $M$ yield
\begin{equation}\label{9}
\{x:Tf(x)>t\}=(\{x:f(x)>t\}\cap (M\cup M^+\cup M^-))\cup(\{x:f^\dagger(x)>t\}\cap (M^+_\dagger\cup M^-_\dagger))
\end{equation}
for $t\in \R$. Using the definitions of $M^-_{\dagger}$ and $M$, together with \eqref{Fvalues}, we obtain
\begin{eqnarray*}
x\in (M^-_{\dagger})^{\dagger}\Leftrightarrow x^{\dagger}\in M^-_{\dagger}
&\Leftrightarrow & x^{\dagger}\in H^-\setminus M~{\text{and}}~F^-(f(x^{\dagger}),f^{\dagger}(x^{\dagger}))=
f^{\dagger}(x^{\dagger})\\
&\Leftrightarrow & x\in H^+\setminus M~{\text{and}}~F^-(f^{\dagger}(x),f(x))=
f(x)\\
&\Leftrightarrow & x\in H^+\setminus M~{\text{and}}~F^+(f(x),f^{\dagger}(x))=
f^{\dagger}(x)\Leftrightarrow x\in M^+_{\dagger}.
\end{eqnarray*}
Consequently, $(M^-_{\dagger})^{\dagger}=M^+_{\dagger}$ and
$(M^+_{\dagger})^{\dagger}=M^-_{\dagger}$.  Therefore
$$\left(\{x:f^\dagger(x)>t\}\cap (M^+_\dagger\cup M^-_\dagger)\right)^\dagger=\{x:f(x)>t\}\cap (M^-_\dagger\cup M^+_\dagger).$$
In particular,
$${\mathcal{H}}^{n}\left(\left(\{x:f^\dagger(x)>t\}\cap (M^+_\dagger\cup M^-_\dagger)\right)\right)={\mathcal{H}}^{n}\left(\{x:f(x)>t\}\cap (M^-_\dagger\cup M^+_\dagger)\right).$$
It follows from (\ref{9}) that ${\mathcal{H}}^{n}(\{x:Tf(x)>t\})={\mathcal{H}}^{n}(\{x:f(x)>t\})$, so $T$ is equimeasurable.  This completes the proof when $X={\mathcal{M}}(\R^n)$ or ${\mathcal{S}}(\R^n)$.

Now suppose that $X={\mathcal{M}}_+(\R^n)$ or ${\mathcal{V}}(\R^n)$, so that $D=[0,\infty)$. The second part of the above proof can be applied without change.  If $X={\mathcal{M}}_+(\R^n)$, the first part of the above proof also still applies, but when $X={\mathcal{V}}(\R^n)$, we cannot use constant functions other than $f\equiv 0$ and thus can only obtain the weaker version
\begin{equation}\label{weak}
F^+(0,0)=F^-(0,0)=0
\end{equation}
of (\ref{eq:diagonId}).  Nevertheless, we can follow the argument in the second paragraph when $\min\{r,s\}>0$ and $c=0$, and this yields (\ref{minmax}) when $r,s>0$. Setting $r=s>0$ in (\ref{minmax}) and using (\ref{weak}), we retrieve (\ref{eq:diagonId}) for $r\ge 0$.  With (\ref{eq:diagonId}) in hand, we may assume that $r=0$ and $s>0$.  Then by using (\ref{equi}) with $t=s$ and with $0<t\le s$, one obtains (\ref{minmax}) for $r,s\ge 0$ and the conclusion follows easily as before.
\end{proof}

\begin{cor}\label{corm2}
Let $H\in {{\mathcal{G}}(n,n-1)}$ be oriented, let $X={\mathcal{M}}(\R^n)$, ${\mathcal{M}}_+(\R^n)$, ${\mathcal{S}}(\R^n)$, or $\cV(\R^n)$, and suppose that $T:X\to X$ is pointwise with respect to $H$. If $T$ is equimeasurable, then it maps characteristic functions of sets in ${\mathcal{L}}^n$ to characteristic functions of sets in ${\mathcal{L}}^n$ and $\di_T=\Id$, $\di_T=\dagger$, $\di_T=\di_{P_H}$, or $\di_T=\di_{P_H}^{\dagger}$, where $\Id$ is the identity map and $\dagger$ is reflection in $H$.
\end{cor}

\begin{proof}
By Theorem~\ref{thmm1}, \eqref{Fvalues} holds, implying that $F^\pm(0,0)=0$, $F^\pm(1,1)=1$,
$$\text{either }F^+(1,0)=1\text{ and }F^-(0,1)=0,\text{ or }F^+(1,0)=0\text{ and } F^-(0,1)=1,$$
and
$$\text{either }F^+(0,1)=0\text{ and }F^-(1,0)=1\text{ or }F^+(0,1)=1\text{ and }F^-(1,0)=0.$$
Therefore we can have the following four combinations: (i) $F^+(r,s)=F^-(r,s)=r$ for $r,s\in\{0,1\}$, (ii) $F^+(r,s)=\max\{r,s\}$ and $F^-(r,s)=\min\{r,s\}$ for $r,s\in\{0,1\}$, (iii) $F^+(r,s)=\min\{r,s\}$ and $F^-(r,s)=\max\{r,s\}$ for $r,s\in\{0,1\}$, or (iv) $F^+(r,s)=F^-(r,s)=s$ for $r,s\in\{0,1\}$.  These correspond to $T1_A=1_A$, $T1_A=P_H1_A$, $T1_A=(P_H1_A)^{\dagger}$, and $T_A=(1_A)^{\dagger}$, each for all $A\in {\mathcal{L}}^n$, respectively.  In particular, $T$ maps characteristic functions of sets in ${\mathcal{L}}^n$ to characteristic functions of sets in ${\mathcal{L}}^n$, and $\di_T$ is $\Id$, $\di_{P_H}$, $\di{P_H^{\dagger}}$, or $\dagger$.
\end{proof}

Despite the previous result, maps $T$ that are both pointwise and equimeasurable need not be one of the four special maps, $T=\Id$, $T=\dagger$, $T=P_H$, or $T=P_H^{\dagger}$.  Indeed, by Theorem~\ref{thmm1}, it is enough to define $T$ via associated functions $F^+$ and $F^-$ that satisfy (\ref{Fvalues}).  For example, one can take
$$
F^+(r,s)=\begin{cases}
r,& {\text{if $r\in \Q\cap D$}},\\
s,& {\text{if $r\in D\setminus\Q$}}\\
\end{cases}
~~\quad~~{\text{and}}~~\quad~~
F^-(s,r)=\begin{cases}
s,& {\text{if $r\in \Q\cap D$}},\\
r,& {\text{if $r\in D\setminus\Q$}}.\\
\end{cases}
$$
The next few results supply further conditions that eliminate such exotic examples.

\begin{lem}\label{lemsept16}
Let $H\in {{\mathcal{G}}(n,n-1)}$ be oriented, let $X={\mathcal{M}}(\R^n)$, ${\mathcal{M}}_+(\R^n)$, ${\mathcal{S}}(\R^n)$, or $\cV(\R^n)$, and suppose that $T:X\to X$ is pointwise with respect to $H$ and equimeasurable. If the functions $F^+$ and $F^-$ associated with $T$ are continuous on $D^2$, then $T=\Id$, $T=\dagger$, $T=P_H$, or $T=P_H^{\dagger}$.
\end{lem}

\begin{proof}
As $F^+$ is continuous on $D^2$, \eqref{Fvalues} implies that either $F^+(r,s)=r$ for $(r,s)\in D^2$ or $F^+(r,s)=s$ for $(r,s)\in D^2$.  Let
\begin{equation*}
E_1=\{(r,s)\in D^2: r\ge s\}~\quad{\text{and}}~\quad E_2=\{(r,s)\in D^2: r\le s\}.
\end{equation*}
As $E_1$ is connected, $\{(r,s)\in E_1: F^+(r,s)=r\}$ is either empty or $E_1$. If it is empty, then $F^+(r,s)=s$ on $E_1$. It follows that either $F^+(r,s)=r=\max\{r,s\}$ on $E_1$ or $F^+(r,s)=s=\min\{r,s\}$ on $E_1$. In the same way, either $F^+(r,s)=\max\{r,s\}$ on $E_2$ or $F^+(r,s)=\min\{r,s\}$ on $E_2$.  Similar arguments show that the same possibilities hold when $F^+$ is replaced by $F^-$.  Taking \eqref{Fvalues} into account, we arrive at four possibilities for $F^+$ and $F^-$ on $D^2$, corresponding to those for $T$ in the statement of the corollary.
\end{proof}

Motivated by the previous lemma, we now seek conditions ensuring that the associated functions $F^+$ and $F^-$ are continuous on $D^2$.

\begin{lem}\label{thmm4}
Let $H\in {{\mathcal{G}}(n,n-1)}$ be oriented, let $X={\mathcal{M}}(\R^n)$, ${\mathcal{M}}_+(\R^n)$, ${\mathcal{S}}(\R^n)$, or $\cV(\R^n)$, and suppose that $T:X\to X$ is pointwise with respect to $H$ with associated functions $F^+$ and $F^-$.  Then

\noindent{\rm{(i)}}  $T$ is monotonic if and only if $F^+$ and $F^-$ are increasing in each variable, and

\noindent{\rm{(ii)}}  if $T$ is a rearrangement, then $F^+$ and $F^-$ are continuous on $D^2$.
\end{lem}

\begin{proof}
(i) It follows from the definition of a pointwise map that $T$ is monotonic if $F^+$ and $F^-$ are increasing in each variable. For the other implication, suppose that $T$ is monotonic.  Let $r_1,s_1,r_2,s_2\in D$ satisfy $r_1\le r_2$ and $s_1\le s_2$. For $i=1,2$, define $f_i\in X$ by
\begin{align}\label{eqExFct}
f_i(x)=r_i1_{B^n\cap H^+}(x)+s_i1_{B^n\cap H^-}(x)+\min\{0,r_i,s_i\}1_{\R^n\setminus B^n}(x)
\end{align}
for $x\in \R^n$.  (The last term in (\ref{eqExFct}) is to ensure that $f_i\in X$ when $X={\mathcal{S}}(\R^n)$ or ${\mathcal{V}}(\R^n)$.)  For $x\in B^n\cap H^+$ we have
$$Tf_i(x)=F^+(f_i(x),f_i^{\dagger}(x))=F^+(r_i,s_i)$$
for $i=1,2$.
Since $f_1\le f_2$ and $T$ is monotonic, $Tf_1(x)\le Tf_2(x)$ for almost all $x\in B^n\cap H^+$ and hence $F^+(r_1,s_1)\le F^+(r_2,s_2)$.  A similar argument holds for $F^-$.  It follows that $F^+$ and $F^-$ are increasing in each variable.

(ii) Suppose that $T$ is a rearrangement, i.e., equimeasurable and monotonic. Then \eqref{Fvalues} holds. Let $(r_k,s_k)$, $k\in \N$, be a sequence in $D^2$ converging to $(r,s)$.  We may assume that $r\ne s$, as otherwise \eqref{Fvalues} implies $F^\pm(r_k,s_k)\to F^\pm(r,s)$ as $k\to\infty$. Without loss of generality, suppose that $r>s$. By considering subsequences, we may also assume that $\{(r_k,s_k): k\in \N\}$ is contained in one of the four sets
\begin{align*}
D_{++}=\{(r',s')\in D^2: r'\ge r,s'\ge s\},\quad D_{+-}=\{(r',s')\in D^2: r'\ge r,s'\le s\},\\
D_{-+}=\{(r',s')\in D^2: r'\le r,s'\ge s\},\quad D_{--}=\{(r',s')\in D^2: r'\le r,s'\le s\}.
\end{align*}
If $\{(r_k,s_k): k\in \N\}\subset D_{++}$ and $F^+(r,s)=r$, then (i) implies that for sufficiently large $k$, we have
$$
r= F^+(r,s)\le  F^+(r_k,s_k)\le F^+(r_k,r_k)=r_k.
$$
Since $r_k\to r$, this shows that $F^+(r_k,s_k)\to F^+(r,s)$, and then $F^-(r_k,s_k)\to F^-(r,s)$ as $k\to\infty$ by \eqref{Fvalues}. If $\{(r_k,s_k): k\in \N\}\subset D_{++}$ and $F^+(r,s)=s$, we have $F^-(r,s)=r$ by \eqref{Fvalues}, and the same arguments can be applied to $F^-$.

Suppose that $\{(r_k,s_k): k\in \N\}\subset D_{+-}$.  Then $(r,s_k)\in D_{--}$ and $(r_k,s)\in D_{++}$, so both sides of
$$
F^+(r,s_k)\le  F^+(r_k,s_k)\le F^+(r_k,s),
$$
converge to $F^+(r,s)$ and it follows that $F^+(r_k,s_k)\to F^+(r,s)$ as $k\to\infty$.  Equation \eqref{Fvalues} now yields $F^-(r_k,s_k)\to F^-(r,s)$ as $k\to \infty$.  The remaining case $\{(r_k,s_k): k\in \N\}\subset D_{-+}$ is treated in a similar way.
\end{proof}

\begin{lem}\label{lemcorm3}
Let $H\in {{\mathcal{G}}(n,n-1)}$ be oriented, let $X={\mathcal{M}}(\R^n)$ (or ${\mathcal{M}}_+(\R^n)$, ${\mathcal{S}}(\R^n)$, or ${\mathcal{V}}(\R^n)$), and suppose that $T:X\to X$ is pointwise with respect to $H$, equimeasurable, and maps linear functions (or piecewise linear continuous functions, respectively) to continuous functions. Then the functions $F^+$ and $F^-$ associated with $T$ are continuous on $D^2$.
\end{lem}

\begin{proof}
Assume first that $X={\mathcal{M}}(\R^n)$.  Since $T$ is equimeasurable, \eqref{Fvalues} holds, implying that $F^\pm(r,r)=r$ for $r\in D$.  Suppose that $T$ maps linear functions to continuous functions. We may assume that $H=e_n^{\perp}$ and define
$f(x)=x\cdot(e_1+e_n)$ for $x\in \R^n$.  Then $f$ is linear and if $x=(x_1,\dots,x_n)$, then $f(x)=x_1+x_n$ and $f^{\dagger}(x)=x_1-x_n$.  Thus for $x\in H^+$, we have $Tf(x)=F^+(f(x),f^{\dagger}(x))=F^+(x_1+x_n,x_1-x_n)$.  Let $x_{rs}=((r+s)/2,0,\dots,0,(r-s)/2)$ and note that $x_{rs}\in H^+$ if and only if $(r,s)\in E_1=\{(r,s)\in D^2: r\ge s\}$.  Consequently,
\begin{align}\label{eq:linFct}
(r,s)\mapsto Tf(x_{rs})=F^+(r,s)
\end{align}
is continuous on $E_1$. A similar argument using the linear function $f(x)=x\cdot(e_1-e_n)$ shows that $F^+$ is continuous on $E_2=\{(r,s)\in D^2: r\le s\}$.  It follows that $F^+$ is continuous on $D^2$, and we arrive at the same conclusion for $F^-$ similarly.

Suppose that $X={\mathcal{M}}_+(\R^n)$, ${\mathcal{S}}(\R^n)$, or $X={\mathcal{V}}(\R^n)$. We adopt the notation of the first part of this proof. For $(r_0,s_0)\in E_1$ we can choose a nonnegative piecewise linear continuous function $f$ coinciding with $x\mapsto |x\cdot(e_1+ e_n)|\ge 0$ on the ball $tB^n$ with $t>{2}^{-1/2}\|(r_0,s_0)\|$ and vanishing outside an even larger ball. Clearly $f\in X$.  Following the arguments above, it can be seen that the restriction of $(r,s)\mapsto F^+(r,s)$ to $E_1$ is continuous at $(r_0,s_0)$. Similar arguments for $(r_0,s_0)\in E_2$ and for $F^-$ lead to the desired conclusion.
\end{proof}

Let $1\le p\le \infty$.  A function $F:D^2\to \R^2$ is \emph{$l^p_2$-contracting} if
\begin{equation*}%\label{eqLast}
\| F(r,s)-F(r',s')\|_p\le \| (r,s)-(r',s')\|_p
\end{equation*}
for $(r,s),(r',s')\in D^2$, where $\|\cdot\|_p$ is the norm in $l^p_2$.
For instance, the function
\begin{equation}\label{eqLast}
F(r,s)=(M_{\infty}(r,s),M_{-\infty}(s,r)),
\end{equation}
with the associated functions of the polarization operation as components, is $l^p_2$-contracting. In fact, this can be checked directly for $p=\infty$.  For $1\le p<\infty$, it follows from the inequality
$$|r-s'|^p+|s-r'|^p\le |r-r'|^p+|s-s'|^p,$$
where $r\le s$ and $s'\le r'$, which is in turn a consequence of the convexity of the function $|t|^p$, $p\ge 1$.  Indeed, as in \cite[p.~43]{VSPhD}, the latter implies that if $a\in \R$ and $b,c\ge 0$, then
$$|a|^p-|a-b|^p\le |a+c|^p-|a-b+c|^p,$$
so the required inequality results from setting $a=s-r'$, $b=s-r$, and $c=r'-s'$.

\begin{lem}\label{lemLp}
Let $H\in {{\mathcal{G}}(n,n-1)}$ be oriented, let $X={\mathcal{M}}(\R^n)$, ${\mathcal{M}}_+(\R^n)$, ${\mathcal{S}}(\R^n)$, or $\cV(\R^n)$, and let $1\le p\le \infty$. Suppose that $T:X\to X$ is pointwise with respect to $H$ with associated functions $F^+$ and $F^-$.  Then

\noindent{\rm{(i)}}  $T$ is $L^p$-contracting if and only if \begin{equation}\label{mar61}
F(s,t)=(F^+(s,t),F^-(t,s)),~~\quad(s,t)\in D^2,
\end{equation}
is $l^p_2$-contracting, and

\noindent{\rm{(ii)}} if $T$ is $L^p$-contracting, then $F^+$ and $F^-$ are Lipschitz on $D^2$ with Lipschitz constant $\sqrt2$.
\end{lem}

\begin{proof}
(i)  We first show that if $F$ is $l^p_2$-contracting, then $T$ is $L^p$-contracting.  Let $f_1, f_2\in X$.  Suppose that $p<\infty$.  As $T$ is pointwise, we have
\begin{align*}
\|Tf_1-Tf_2\|_p^p&=\int_{\R^n}|Tf_1(x)-Tf_2(x)|^p\,dx \\&=\int_{H^+}\big|F^+\big(f_1(x),f_1^\dagger(x)\big)-F^+
\big(f_2(x),f_2^\dagger(x)\big)\big|^pdx\\&\quad +\int_{H^-}\big|F^-\big(f_1(x),f_1^\dagger(x)\big)-
F^-\big(f_2(x),f_2^\dagger(x)\big)\big|^p\,dx.
\end{align*}
Substituting $x$ by $x^\dagger$ in the second integral and using the definition of $F$ in (\ref{mar61}), we obtain \begin{align}\label{eqCOntr1}
\|Tf_1-Tf_2\|_p^p=\int_{H^+}\big\|F\big(f_1(x),f_1^\dagger(x)\big)-
F\big(f_2(x),f_2^\dagger(x)\big)\big\|_p^p\,dx.
\end{align}
In a similar fashion, it is easily seen that
\begin{align}\label{eqCOntr2}
\|f_1-f_2\|_p^p=\int_{H^+}\big\|\big(f_1(x),f_1^\dagger(x)\big)-
\big(f_2(x),f_2^\dagger(x)\big)\big\|_p^p\,dx.
\end{align}
Now if $F$ is $l^p_2$-contracting, \eqref{eqCOntr1} is bounded from above by \eqref{eqCOntr2}, so $T$ is  $L^p$-contracting.

The case when $p=\infty$ follows similarly from the equations
$$
\|Tf_1-Tf_2\|_\infty=\esssup_{x\in H^+}\big\|F\big(f_1(x),f_1^\dagger(x)\big)-F\big(f_2(x),f_2^\dagger(x)\big)
\big\|_\infty
$$
and
$$
\|f_1-f_2\|_\infty=\esssup_{x\in H^+}\big\|\big(f_1(x),f_1^\dagger(x)\big)-\big(f_2(x),f_2^\dagger(x)\big)
\big\|_\infty.
$$

To show the other direction, let $r_1,s_1,r_2,s_2\in D$ and define $f_i\in X$, $i=1,2$, by \eqref{eqExFct} with the term $\min\{0,r_i,s_i\}$ replaced by $\min\{0,r_1,r_2,s_1,s_2\}$.  As $T$ is  $L^p$-contracting, we get
$$
\frac{\kappa_n}{2}\|F(r_1,s_1)-F(r_2,s_2)\|_p^p\le
	\|Tf_1-Tf_2\|_p^p\le \|f_1-f_2\|_p^p=\frac{\kappa_n}{2}\|(r_1,s_1)-(r_2,s_2)\|_p^p,
$$
so $F$ is $l^p_2$-contracting.

(ii) Suppose that $T$ is $L^p$-contracting.  Then $F$ is $l^p_2$-contracting by (i), so
\begin{align*}
|F^+(r_1,s_1)-F^+(r_2,s_2)|&\le
\|F(r_1,s_1)-F(r_2,s_2)\|_p\le \|(r_1,s_1)-(r_2,s_2)\|_p\\&\le \sqrt 2\|(r_1,s_1)-(r_2,s_2)\|,
\end{align*}
as $\|\cdot\|_p\le \|\cdot\|_1\le \sqrt2 \|\cdot\|$. Hence $F^+$ is Lipschitz on $D^2$ with Lipschitz constant $\sqrt{2}$, and the same argument can be applied to $F^-$.
\end{proof}

\begin{lem}\label{lemcorm5}
Let $H\in {{\mathcal{G}}(n,n-1)}$ be oriented, let $X={\mathcal{M}}(\R^n)$, ${\mathcal{M}}_+(\R^n)$, ${\mathcal{S}}(\R^n)$, or $\cV(\R^n)$, and suppose that $T:X\to X$ is pointwise with respect to $H$ with associated functions $F^+$ and $F^-$.  If $T$ reduces the modulus of continuity, then $T$ is $L^\infty$-contracting.
\end{lem}

\begin{proof}
Suppose that $T$ reduces the modulus of continuity and let $r,s,r',s'\in D$.
Choose $f\in X$ and $x,y\in H^+$ such that $f(x)=r$, $f(y)=r'$, $f(x^{\dagger})=s$, $f(y^{\dagger})=s'$, and
$$
\omega_d(f)=\|(f(x),f(x^\dagger))-(f(y),f(y^\dagger))\|_\infty=
\|(r,s)-(r',s')\|_\infty,
$$
where $d=\|x-y\|$. As $T$ reduces the modulus of continuity,
\begin{equation}\label{mar62}
|F^+(r,s)-F^+(r',s')|=|Tf(x)-Tf(y)|\le \omega_d(Tf)\le \omega_d(f)=\|(r,s)-(r',s')\|_\infty.
\end{equation}
A similar relation for $|Tf(x^\dagger)-Tf(y^\dagger)|$ yields
\begin{equation}\label{mar63}
|F^-(s,r)-F^-(s',r')|\le \|(r,s)-(r',s')\|_\infty.
\end{equation}
From (\ref{mar62}) and (\ref{mar63}), we conclude that $F$ (defined by (\ref{mar61})) is $l^\infty_2$-contracting and the result follows from Lemma~\ref{lemLp}(i) with $p=\infty$.
\end{proof}

Summarizing, we have the following set of characterizations.

\begin{thm}\label{thmport}
Let $H\in {{\mathcal{G}}(n,n-1)}$ be oriented, let $X={\mathcal{M}}(\R^n)$ (or ${\mathcal{M}}_+(\R^n)$, ${\mathcal{S}}(\R^n)$, or $\cV(\R^n)$), and suppose that $T:X\to X$ is pointwise with respect to $H$ and equimeasurable.  The following statements are equivalent.

\noindent{\rm{(i)}} The associated functions $F^+$ and $F^-$ are continuous on $D^2$.

\noindent{\rm{(ii)}} $T$ is monotonic.

\noindent{\rm{(iii)}} $T$ is a rearrangement.

\noindent{\rm{(iv)}}  $T$ maps linear functions (or piecewise linear continuous functions, respectively) to continuous functions.

\noindent{\rm{(v)}} $T$ is $L^p$-contracting for some (or, equivalently, for all) $1\le p\le\infty$.

\noindent{\rm{(vi)}} $T$ reduces the modulus of continuity.

\noindent{\rm{(vii)}} $T=\Id$, $T=\dagger$, $T=P_H$, or $T=P_H^{\dagger}$.
\end{thm}

\begin{proof}
It is easy to check that (vii) implies (i)--(iv) and (vi). To see that (vii)$\Rightarrow$(v), it is enough to apply Lemma~\ref{lemLp} with $T=P_H$, since the function $F$ in \eqref{eqLast} is $l^p_2$-contracting. Lemma~\ref{lemsept16} gives (i)$\Rightarrow$(vii) and  Lemmas~\ref{thmm4}(ii), \ref{lemcorm3}, and \ref{lemLp}(ii) show that (iii)$\Rightarrow$(i), (iv)$\Rightarrow$(i), and (v)$\Rightarrow$(i), respectively. That (ii)$\Rightarrow$(iii) follows from the definition of a rearrangement.  Finally, the implication (vi)$\Rightarrow$(v) follows from Lemma~\ref{lemcorm5} and the implications already established.
\end{proof}

\section{General maps between sets and between functions}\label{General sets}

The proof of the following result is essentially the same as that of \cite[Theorem~10.1(i)]{BGG}.  The statement is more general, since the map $\di$ need not be an $i$-symmetrization in the sense of \cite{BGG}, the sets concerned need not be compact, and invariance on $H$-symmetric cylinders is replaced by the weaker condition that $\di$ respects $H$-cylinders.

\begin{lem}\label{SteinerCompact}
Let $i\in \{1,\dots,n-1\}$ and let $H\in {\mathcal{G}}(n,i)$.  Suppose that $\di:{\mathcal{E}}\subset {\mathcal{L}}^n\rightarrow {\mathcal{L}}^n$ is monotonic, measure preserving, and respects $H$-cylinders.  Then
\begin{equation}\label{mainineq}
{\mathcal{H}}^{n-i}\left((\di K)\cap (H^{\perp}+x)\right)={\mathcal{H}}^{n-i}\left(K\cap (H^{\perp}+x)\right)
\end{equation}
for $K\in {\mathcal{E}}$ and ${\mathcal{H}}^i$-almost all $x\in H$.
\end{lem}

In the following results, we always assume for convenience that ${\mathcal{K}}^n_n\subset{\mathcal{E}}$, even though this assumption can sometimes be weakened.

\begin{lem}\label{lemm5}
Let $H=u^{\perp}$, $u\in S^{n-1}$, let ${\mathcal{K}}^n_n\subset{\mathcal{E}}\subset {\mathcal{L}}^n$, and suppose that $\di:{\mathcal{E}}\rightarrow {\mathcal{L}}^n$ is monotonic, respects $H$-cylinders, and maps balls to balls.  Then there is a contraction (i.e., a Lipschitz function with Lipschitz constant $1$) $\varphi_{\di}:\R\to\R$ such that
\begin{equation}\label{ball}
\di B(x+tu,r)=B(x+\varphi_{\di}(t)u,r),
\end{equation}
essentially, for $r>0$ and $x\in H$.
\end{lem}

\begin{proof}
In the proof we ignore sets of ${\mathcal{H}}^n$-measure zero.  For $r>0$ and $x\in H$, let $C(x,r)$ be the infinite spherical cylinder of radius $r$ and axis $H^{\perp}+x$. Let $t\in \R$. By our assumptions, $\di B(x+tu,r)$ is a ball contained in $C(x,r)$.  We claim that the radius of $\di B(x+tu,r)$ is $r$.  Indeed, if it is less than $r$, there is a ball $B_0\subset B(x+tu,r)$ whose projection on $H$ is disjoint from that of $\di B(x+tu,r)$.  However, by the monotonicity of $\di$, $\di B_0\subset \di B(x+tu,r)$, contradicting the fact that $\di$ respects $H$-cylinders.  It follows that $\di B(x+tu,r)=B(x+t'u,r)$ for some $t'=t'(r,t,x)\in \R$.

Fix $t\in \R$ and suppose that neither of the balls $B(x_j+tu,r_j)$, $j=1,2$, contains the other.  Let $z\in H$ and $s\in \R$ be such that $B(z+tu,s)\supset B(x_1+tu,r_1)\cup B(x_2+tu,r_2)$ is tangent to both $B(x_1+tu,r_1)$ and $B(x_2+tu,r_2)$ at points in $\partial C(z,s)\cap (H+tu)$.  Then there are $t'=t'(s,t,z)$ and $t_j'=t'(r_j,t,x_j)$, $j=1,2$, with
\begin{align*}
B(z+t'u,s)&=\di B(z+tu,s)\supset \di B(x_1+tu,r_1)\cup \di B(x_2+tu,r_2)\\
&= B(x_1+t_1'u,r_1)\cup  B(x_2+t_2'u,r_2),
\end{align*}
where we used the monotonicity of $\di$.  It follows that $B(z+t'u,s)$ contains $B(x_1+t'_1u,r_1)$ and $B(x_2+t'_2u,r_2)$ and is tangent to both of them at points in $\partial C(z,s)\cap (H+t'u)$. This forces $t_1'=t'=t_2'$,  so
\begin{align}\label{eq13}
t'(r_1,t,x_1)=t'(r_2,t,x_2).
\end{align}
If one of $B(x_j+tu,r_j)$, $j=1,2$, is contained in the other, say $B(x_1+tu,r_1)\subset B(x_2+tu,r_2)$, choose $B(z+tu,s)$ disjoint from $B(x_2+tu,r_2)$. Then \eqref{eq13}, applied first to the disjoint balls  $B(x_1+tu,r_1)$ and $B(z+tu,s)$, and then to the disjoint balls  $B(x_2+tu,r_2)$ and $B(z+tu,s)$, yields
$$t'(r_1,t,x_1)=t'(s,t,z)=t'(r_2,t,x_2).$$
This shows that \eqref{eq13} holds generally, so $t'(r,t,x)$ is independent of $r$ and $x$. Then $\varphi_{\di}(t)=t'$ is the required function.

Suppose that there are $s,t\in \R$ and $\ee>0$ such that
$$|\varphi_{\di}(s)-\varphi_{\di}(t)|=|s-t|+\ee.$$
Let $|s-t|/2<r<(|s-t|+\ee)/2$.  If $K=B(su,r)\cap B(tu,r)$, then there is a ball $B_1\subset K$ and hence $\emptyset\neq\di B_1\subset \di K$.  On the other hand, since $K\in {\mathcal{E}}$, the monotonicity of $\di$ also implies that
$$\di K\subset \di B(su,r)\cap \di B(tu,r)=B(\varphi_{\di}(s)u,r)\cap B(\varphi_{\di}(t)u,r)=\emptyset.$$
This contradiction shows that $\varphi_{\di}$ is Lipschitz with Lipschitz constant 1.
\end{proof}

\begin{lem}\label{lemm6}
Let $H=u^{\perp}$, $u\in S^{n-1}$, let ${\mathcal{K}}^n_n\subset{\mathcal{E}}\subset {\mathcal{L}}^n$, and suppose that $\di:{\mathcal{E}}\rightarrow {\mathcal{L}}^n$ is monotonic, measure preserving, respects $H$-cylinders, and maps balls to balls.  Let $\varphi_{\di}:\R\to \R$ be the function from Lemma~\ref{lemm5}.  Then for each $H$-symmetric $K\in {\mathcal{K}}^n_n$ and $t\in \R$, we have
\begin{equation}\label{hsym}
\di(K+tu)=K+\varphi_{\di}(t)u,
\end{equation}
essentially.
\end{lem}

\begin{proof}
In the proof we ignore sets of ${\mathcal{H}}^n$-measure zero.  If $r>0$ and $x\in H$, let $C(x,r)$ be the infinite spherical cylinder of radius $r$ and axis $H^{\perp}+x$.  Suppose first that $K=(B(x,r)\cap H)+[-su,su]$ is an $H$-symmetric spherical cylinder, where $r,s>0$ and $x\in H$.  Let $t\in \R$ and let $S=\cup_{m\in\N}\, B(z_m+tu,s)$, where $\{z_m: m\in \N\}$ is dense in $B(x,r)\cap H$.  Then
$$\inte (K+tu)\subset S\cap C(x,r)\subset K+tu$$
and we can write $K+tu=S\cap C(x,r)$ since we are ignoring sets of $\cH^n$-measure zero.  By the monotonicity of $\di$,
\begin{equation}\label{cy1}
S'=\cup_{m\in \N}\,\di B(z_m+tu,s)\subset \di S,
\end{equation}
as there are countably many sets in the union. By Lemma~\ref{lemm5},
$$S'=\cup_{m\in\N}\,B(z_m+\varphi_{\di}(t)u,s)=S+(\varphi_{\di}(t)-t)u.$$
This and the measure-preserving property of $\di$ yield ${\mathcal{H}}^n(S')={\mathcal{H}}^n(S)={\mathcal{H}}^n(\di S)$ and therefore (\ref{cy1}) implies that $\di S=S'$.  Since $\di$ respects $H$-cylinders, we have $\di (K+tu)\subset C(x,r)$ and hence
$$\di (K+tu)\subset \di S\cap C(x,r)=(S+(\varphi_{\di}(t)-t)u)\cap C(x,r)=K+\varphi_{\di}(t)u.$$

Now let $K\in {\mathcal{K}}^n_n$ be an arbitrary $H$-symmetric set.  Let $\{x_m: m\in \N\}$ be a dense set in $K|H$.  Since $K+tu$ is symmetric with respect to $H+tu$, it is clear that we can find $H$-symmetric spherical cylinders $C_k=(B(x_{m_k},r_k)\cap H)+[-s_ku,s_ku]$, $k\in \N$, where $r_k,s_k$ are positive rationals, such that
\begin{equation}\label{cco}
\inte(K+tu)\subset \cup_k (C_k+tu)\subset K+tu.
\end{equation}
By the previous paragraph, $\di (C_k+tu)=C_k+\varphi_{\di}(t)u$, so (\ref{cco}) and the monotonicity of $\di$ yields
\begin{equation}\label{cco2}
E=\cup_k\di(C_k+tu)=\cup_k(C_k+\varphi_{\di}(t)u)=
(\cup_kC_k)+\varphi_{\di}(t)u\subset \di (K+tu).
\end{equation}
Since $E$ is a translate of $\cup_kC_k$, (\ref{cco}) and the measure-preserving property of $\di$ imply that ${\mathcal{H}}^n(E)={\mathcal{H}}^n(K+tu)={\mathcal{H}}^n(\di (K+tu))$. Hence, by (\ref{cco}) and (\ref{cco2}), we have
$$\di(K+tu)=E=(\cup_kC_k)+\varphi_{\di}(t)u=K+\varphi_{\di}(t)u,$$
essentially.
\end{proof}

Since $\varphi_{\di}$ is a contraction, we have
\begin{equation}\label{vie}
\varphi_{\di}(0)=0 ~~\quad\Leftrightarrow~~\quad |\varphi_{\di}(t)|\le |t|,\text{ for } t\in \R.
\end{equation}
Indeed, the right-hand side follows from the left-hand side on setting $s=0$ in $|\varphi_{\di}(s)-\varphi_{\di}(t)|\le |s-t|$, and the converse is trivial.

\begin{cor}\label{cor6a}
Let $H=u^{\perp}$, $u\in S^{n-1}$, and suppose that $\di:{\mathcal{K}}^n_n \rightarrow {\mathcal{L}}^n$ is monotonic, measure preserving, respects $H$-cylinders, and maps balls to balls.  Let $\varphi_{\di}:\R\to \R$ be the function from Lemma~\ref{lemm5}. Then $\di$ is invariant on $H$-symmetric sets if and only if $\di$ is invariant on $H$-symmetric cylinders, and this occurs if and only if either condition in \eqref{vie} holds.
\end{cor}

\begin{proof}
In the proof we ignore sets of ${\mathcal{H}}^n$-measure zero. Suppose that $\di$ is invariant on $H$-symmetric cylinders.  If $r>0$ and $t\in \R$, then
$$B(tu,r)\subset (B(tu,r)| H)+[-(|t|+r)u,(|t|+r)u]=C,$$
say, an $H$-symmetric cylinder.  Then
$$B(\varphi_{\di}(t)u,r)=\di B(tu,r)\subset \di C=C.$$
This yields
$$[\varphi_{\di}(t)-r,\varphi_{\di}(t)+r]=B(\varphi_{\di}(t)u,r)|H^{\perp}
\subset C|H^{\perp}=[-(|t|+r)u,(|t|+r)u].$$
Therefore $|\varphi_{\di}(t)|\le |t|$.  Then $\varphi_{\di}(0)=0$, so setting $t=0$ in (\ref{hsym}) implies that $\di$ is invariant on $H$-symmetric sets. It follows directly from the definitions that the latter implies that $\di$ is invariant on $H$-symmetric cylinders.
\end{proof}

We shall find it convenient to define, for $K\in \cK^n$, $t\in \R$, and $u\in S^{n-1}$,
\begin{equation}\label{kt}
K_t=(K-tu)\cap (K^{\dagger}+tu).
\end{equation}
Note that if $H=u^{\perp}$, then $K_t$ is $H$-symmetric.

\begin{lem}\label{lemm11}
Let $H=u^{\perp}$, $u\in S^{n-1}$, and let $\varphi:\R\to\R$ be an arbitrary contraction.  Then
\begin{equation}\label{kseg1}
\left(\cup_{t\in\R}(K_t+\varphi(t)u)\right)\cap (H^{\perp}+x)=(K\cap (H^{\perp}+x))+(\varphi(t_x)-t_x)u
\end{equation}
for $K\in \cK^n_n$ and $x\in H$, where $K_t$ is defined by \eqref{kt} and $x+t_xu$ is the midpoint of $K\cap (H^{\perp}+x)$.
\end{lem}

\begin{proof}
If $M\in\cK^n$, $x\in H$, and $s,t\in \R$, then $x+su\in M_t$ if and only if $x+(t\pm s)u\in M$.  Applying this first with $M=K$ and then with $M=K\cap(H^\perp+x)$ shows that
\begin{equation}\label{eq:first}
K_t\cap(H^\perp+x)=(K\cap(H^\perp+x))_t.
\end{equation}
Note also that
\begin{equation}\label{eq:second}
(K\cap(H^\perp+x))_{t_x}=(K\cap(H^\perp+x))-t_xu.
\end{equation}
Let $x\in H$ be such that the (possibly degenerate) line segment $S=K\cap(H^\perp+x)$ is nonempty and let $2r_x\ge0$ be its length. Then $S_{t_x+t}=\emptyset$ for $|t|>r_x$.  Suppose that $|t|\le r_x$. Then $S_{t_x+t}=[x-(r_x-|t|)u,x+(r_x-|t|)u]$.  As $\varphi$ is a contraction,
$|\varphi(t_x+t)-\varphi(t_x)|\le |t|$ and hence
$$S_{t_x+t}+(\varphi(t_x+t)-\varphi(t_x))u\subset S_{t_x+t}+[-|t|u,|t|u]=[x-r_xu,x+r_xu]=S_{t_x}.$$
Rearranging and replacing $t_x+t$ by $t$, we obtain
\begin{align}\label{eqThree}
(K\cap(H^\perp+x))_t+\varphi(t)u=S_t+\varphi(t)u\subset
S_{t_x}+ \varphi(t_x)u=	 (K\cap(H^\perp+x))_{t_x}+\varphi(t_x)u
\end{align}
whenever the left-hand side is nonempty.  Hence \eqref{eqThree} holds for $x\in H$ and $t\in\R$.

Applying, in turn, \eqref{eq:first}, \eqref{eqThree}, and \eqref{eq:second}, we obtain
$$\left(\cup_{t\in\R}(K_t+\varphi(t)u)\right)\cap (H^\perp+x)\subset (K\cap(H^\perp+x))+(\varphi(t_x)-t_x)u.$$
The reverse inclusion is a consequence of
$$ \left(\cup_{t\in\R}(K_t+\varphi(t)u)\right)\cap (H^\perp+x)\supset (K_{t_x}\cap (H^\perp+x))+\varphi(t_x)u$$
together with \eqref{eq:first} and \eqref{eq:second}.  This proves (\ref{kseg1}).
\end{proof}

\begin{thm}\label{thmm7}
Let $H=u^{\perp}$, $u\in S^{n-1}$,
let $\di:{\mathcal{K}}^n_n \rightarrow {\mathcal{L}}^n$, and let $\varphi:\R\to\R$ be a contraction. The following are equivalent.

\noindent{\rm{(i)}} $\di$ is monotonic, measure preserving, respects $H$-cylinders, and maps balls to balls. The mapping $\varphi$ coincides with the function $\varphi_{\di}$ from Lemma~\ref{lemm5}.

\noindent{\rm{(ii)}} For each $K\in \cK^n_n$,
\begin{equation}\label{ktp}
\di K=\cup_{t\in\R}\,\inte(K_t+\varphi(t)u),
\end{equation}
essentially, where $K_t$ is defined by \eqref{kt}.

\noindent{\rm{(iii)}} For each $K\in \cK^n_n$ and ${\mathcal{H}}^{n-1}$-almost all $x\in H$,
\begin{equation}\label{kseg}
(\di K)\cap (H^{\perp}+x)=(K\cap (H^{\perp}+x))+(\varphi(t_x)-t_x)u,
\end{equation}
up to a set of ${\mathcal{H}}^1$-measure zero, where $x+t_xu$ is the midpoint of $K\cap (H^{\perp}+x)$.
\end{thm}

\begin{proof}
(i)$\Rightarrow$(ii).  In the proof we ignore sets of ${\mathcal{H}}^n$-measure zero.  Let $K\in {\mathcal{K}}^n_n$ and let $t\in \R$.  The set $K_t$ defined by (\ref{kt}) is $H$-symmetric, so $K_t+tu\subset K$ is symmetric with respect to $H+tu$.  From the monotonicity of $\di$ and (\ref{hsym}), we obtain
$$\cup_{t\in\Q}\,(K_t+\varphi_{\di}(t)u)=\cup_{t\in\Q}\,\di(K_t+tu)\subset \di K,$$
because there are countably many sets in the union.  Let $t\in \R$ with $K_t\ne \emptyset$ and choose $t_m\in \Q$ with $K_{t_m}\ne \emptyset$ such that $t_m\to t$ as $m\to\infty$.  Since $\varphi_{\di}$ is a Lipschitz map and $t\mapsto K_t$ is continuous on $\{t\in \R:K_t\ne \emptyset\}$, we have $K_{t_m}+\varphi_{\di}(t_m)u\to K_t+\varphi_{\di}(t)u$ in the Hausdorff metric as $m\to \infty$ and hence
$$\inte(K_t+\varphi_{\di}(t)u)\subset \cup_{m\in \N}\,(K_{t_m}+\varphi_{\di}(t_m)u).$$
It follows that
$$U=\cup_{t\in\R}\,\inte(K_t+\varphi_{\di}(t)u)\subset \di K.$$
Let $V=\cup_{x\in (\inte K)|H}\,\inte(K_{t_x}+\varphi_{\di}(t_x)u)$, where $x+t_xu$ is the midpoint of $K\cap (H^{\perp}+x)$.  Then
\begin{equation}\label{v1}
V\subset U\subset \di K.
\end{equation}
If $x\in (\inte K)|H$, then $(K_{t_x}+t_xu)\cap (H^{\perp}+x)=K\cap (H^{\perp}+x)$, so
\begin{eqnarray}\label{eqjan101}
\cH^1(V\cap (H^\perp+x))&\ge & \cH^1\left((\inte(K_{t_x}+\varphi_{\di}(t_x)u))\cap (H^\perp+x)\right)\nonumber\\
&= & \cH^1((K_{t_x}+\varphi_{\di}(t_x)u)\cap (H^\perp+x))\nonumber\\
&=&\cH^1((K_{t_x}+t_xu)\cap (H^\perp+x))= \cH^1(K\cap (H^\perp+x)).
\end{eqnarray}
From this, an application of Fubini's theorem and the measure-preserving property of $\di$ give ${\mathcal{H}}^n(V)\ge {\mathcal{H}}^n(K)={\mathcal{H}}^n(\di K)$.  Therefore, by (\ref{v1}), we have $V=U=\di K$, essentially.

(ii)$\Rightarrow$(iii).  Let $K\in \cK^n_n$ and let $W=\cup_{t\in\R}\,(K_t+\varphi_{\di}(t)u)$.  Then
$$U\cap (H^{\perp}+x)\subset W\cap (H^{\perp}+x)=(K\cap (H^{\perp}+x))+(\varphi(t_x)-t_x)u$$
for $x\in H$, by \eqref{kseg1}.  Moreover, \eqref{v1} and \eqref{eqjan101} imply that
$$\cH^1(U\cap (H^{\perp}+x))\ge \cH^1(V\cap (H^{\perp}+x))\ge\cH^1(K\cap (H^{\perp}+x)),$$
so
\begin{equation}\label{eqjan102}
U\cap (H^{\perp}+x)=(K\cap (H^{\perp}+x))+(\varphi(t_x)-t_x)u,
\end{equation}
up to a set of $\cH^1$-measure zero.  By (ii), $\di K=U$, essentially, and (iii) follows.

(iii)$\Rightarrow$(ii).  Let $K\in \cK^n_n$.  If (iii) holds, then (ii) follows from \eqref{eqjan102} (which did not require (ii) for its proof) and Fubini's theorem.

(ii)$\Rightarrow$(i). Assume that (ii) holds.  Clearly, $\di$ is monotonic due to \eqref{ktp}.  We have already seen that (ii) implies (iii). That $\di$ is {measure} preserving and respects $H$-cylinders follows directly from (\ref{kseg}) and Fubini's theorem. If $z\in H$, $t\in \R$, and $r>0$, then \eqref{ktp} implies that
$$\di B(z+tu,r)\supset B(z+tu,r)_t+\varphi(t)u=B(z+\varphi(t)u,r),$$
essentially.  Together with the measure-preserving property, this proves that $\di$ maps balls to balls.
\end{proof}

None of the properties of $\di$ listed in Theorem~\ref{thmm7}(i) can be omitted.  Indeed, no three of these properties imply the fourth, as is shown by the map $\di'$ from Example~\ref{Marex} below, the map taking a convex body to the $o$-symmetric ball of the same volume, Minkowski symmetrization (see \cite[Section~3]{BGG}), and Example~\ref{thm9ex1} below.

In view of (\ref{kseg}), a map $\di$ satisfying Theorem~\ref{thmm7}(i) may be regarded as a ``still" in a parallel chord movement in the direction $u$, in the sense of convex geometry.  The concept of a parallel chord movement was, in a more general form, introduced by Rogers and Shephard \cite{RS} (see also \cite[p.~543]{Sch14}) and is extremely useful in convex geometry, where, however, it is always assumed that the movement preserves convexity. It is easy to see that $\di$ preserves convexity if and only if $\varphi_{\di}$ is affine; cf.~the proof of Theorem~\ref{thmmar14}.

\begin{cor}\label{corlemExtend}
Let $H=u^{\perp}$, $u\in S^{n-1}$, be oriented with $u\in H^+$ and suppose that ${\mathcal E}={\mathcal K}_n^n$ or ${\mathcal K}^n$.  Suppose that $\di:{\mathcal{E}}\rightarrow {\mathcal{L}}^n$ is monotonic, measure preserving, respects $H$-cylinders, and maps balls to balls.  Let $\varphi_{\di}:\R\to \R$ be the function from Lemma~\ref{lemm5}. Then $\varphi_{\di}(t)=t$, $-t$, $|t|$, or $-|t|$, if and only if $\di$ essentially equals $\mathrm{Id}$, $\dagger$, $\di_{P_H}$, or $\di_{P_H}^\dagger$, respectively.
\end{cor}

\begin{proof}
If $\di$ is given, it is easy to check that $\varphi_{\di}$ has the appropriate form, by applying $\di$ to balls and using (\ref{ball}).  The converse follows directly from (\ref{kseg}).  Indeed, we need only consider $K\in {\mathcal K}_n^n$, for otherwise ${\mathcal{H}}^n(K)=0$ and there is nothing to prove. Consider, for example, the case when $\varphi_{\di}=|t|$; the other cases are similar.  By (\ref{kseg}), for ${\mathcal{H}}^{n-1}$-almost all $x\in H$, we have
\begin{eqnarray*}
(\di K)\cap (H^{\perp}+x)&=&(K\cap (H^{\perp}+x))+(|t_x|-t_x)u\\
&=&\begin{cases}
K\cap (H^{\perp}+x),&\text{if } t_x\ge 0,\\
K^{\dagger}\cap (H^{\perp}+x),&\text{if } t_x<0,
\end{cases}\\
&=&(\di_{P_H}K)\cap (H^{\perp}+x),
\end{eqnarray*}
up to a set of ${\mathcal{H}}^1$-measure zero.  This shows that $\di$ essentially equals $\di_{P_H}$, as required.
\end{proof}

\begin{ex}\label{Marex}
{\rm  Let $H=u^{\perp}$, $u\in S^{n-1}$, be oriented with $u\in H^+$.  Define $\di':\cL^n\to\cL^n$ by $\di'E=E^+\cup E^-$, where
$E^+=E\cap H^+$ and $E^-\subset H^-$ is given by
$$E^-\cap (H^\perp+x)=[x,x-\lambda u],$$
with $\lambda=\cH^1(E\cap H^-\cap(H^\perp+x))$ for each $x\in H$.  Thus $\di'=\Id$ when applied to subsets of  $H^+$ and $\di'$ corresponds to Blaschke shaking \cite[Note~2.4]{Gar06} with respect to $H$  when applied to subsets of  $H^-$.

Define $\di:\cL^n\to\cL^n$ by $\di=\di'\circ \di_{P_H}$. It is easy to see that $\di'$ is monotonic, measure preserving, and respects $H$-cylinders, and hence $\di$ also has these three properties.

We claim that $\di=\di_{P_H}$ on $\cK^n$.  To see this, let $K\in {\mathcal K}^n$. Then
$$(\di_{P_H}K)\cap H^-=(K\cap K^\dagger)\cap H^-.$$
As $K\cap K^\dagger$ is $H$-symmetric and convex, we have
$$\left((K\cap K^\dagger)\cap H^-\right)\cap (H^\perp+x)=[x,x-\lambda u],$$
where $\lambda=\cH^1((K\cap K^\dagger)\cap H^-\cap(H^\perp+x))$ for $x\in H$. Thus, $\di K= \di'(\di_{P_H}K)=\di_{P_H}K$, proving the claim.

As a consequence, $\di$ maps balls to balls and therefore satisfies all the hypotheses of Corollary~\ref{corlemExtend} with $\varphi_{\di}(t)=|t|$, yet it is clear that $\di$ is essentially different from $\di_{P_H}$.  Indeed, Corollary~\ref{corlemExtend} is false for maps $\di: \cK^n_n\subset\cE\to\cL^n$ if $\cE$ contains an $H$-symmetric union of two disjoint balls, since if $E$ is such a union, then $\di E\neq \di_{P_H}E$.\qed}
\end{ex}

\begin{lem}\label{ballmap}
Let $H\in {\mathcal{G}}(n,n-1)$, let ${\mathcal{K}}^n_n\subset{\mathcal{E}}\subset {\mathcal{L}}^n$, and suppose that $\di:{\mathcal{E}}\rightarrow {\mathcal{L}}^n$ is measure preserving and perimeter preserving on convex bodies. Then $\di$ maps balls to balls.
\end{lem}

\begin{proof}
Let $x\in \R^n$ and let $r>0$.  Our assumptions imply that $\di B(x,r)$ is a set of finite perimeter that has the same ${\mathcal{H}}^n$-measure and perimeter as $B(x,r)$.  It follows from the isoperimetric inequality for sets of finite perimeter and its equality condition (see \cite{DeG} or \cite[Theorem~14.1]{Mag}) that $\di B(x,r)$ is a ball, modulo a set of ${\mathcal{H}}^{n}$-measure zero.
\end{proof}

\begin{thm}\label{thmm9convex}
Let $H=u^{\perp}$, $u\in S^{n-1}$, be oriented with $u\in H^+$ and suppose $\di:{\mathcal K}_n^n\rightarrow {\mathcal{L}}^n$ is monotonic, measure preserving, respects $H$-cylinders, and perimeter preserving on convex bodies. If $\varphi_\di$ is the contraction defined in Lemma~\ref{lemm5}, then $|\varphi_{\di}'(t)|=1$ for ${\mathcal{H}}^{1}$-almost all $t\in \R$.

Conversely, if $\varphi:\R\to \R$ is a contraction satisfying $|\varphi'(t)|=1$ for ${\mathcal{H}}^{1}$-almost all $t\in \R$, then \eqref{kseg} defines a map $\di:{\mathcal K}^n\rightarrow {\mathcal{L}}^n$ that is monotonic, measure preserving, respects $H$-cylinders, and perimeter preserving on convex bodies.
\end{thm}

\begin{proof}
We may assume that $u=e_n$ and write $H^{\perp}=\langle e_n\rangle$ for the $x_n$-axis. If $x\in H=e_n^{\perp}$, write $x=(x_1,\dots,x_{n-1})$. If $t\in \R$ and $r>0$, let $D(t,r)=B(te_1,r)\cap e_n^{\perp}$.

Suppose that $t>0$. For $0<r<t$, define
$$K(t,r)=\{x+x_ne_n: x\in D(t,r),~~ 0\le x_n\le 2x_1\}.$$
Then $K(t,r)\in {\mathcal{K}}^n_n$ and if $x\in D(t,r)$, then $x+x_1e_n$ is the midpoint of $K(t,r)\cap (\langle e_n\rangle+x)$. Lemma~\ref{ballmap} implies that $\di$ maps balls to balls, so (\ref{kseg}) holds for $\di$.  Hence
\begin{equation}\label{did}
\di K(t,r)=\{x+x_ne_n: x\in D(t,r),~~ \varphi_{\di}(x_1)-x_1\le x_n\le \varphi_{\di}(x_1)+x_1\},
\end{equation}
essentially. Since $\varphi_{\di}$ is Lipschitz by Lemma~\ref{lemm5}, $\di K(t,r)$ is a set of finite perimeter; see \cite[Proposition~3.62]{AFP}. From (\ref{did}) (or see (\ref{kseg})), we have
$$(\di K(t,r))\cap (\langle e_n\rangle+x)=(K(t,r)\cap (\langle e_n\rangle+x))+(\varphi_\di(x_1)-x_1)e_n,$$
up to a set of ${\mathcal{H}}^{1}$-measure zero, for ${\mathcal{H}}^{n-1}$-almost all $x\in D(t,r)$.  Also, $\|\nabla(2x_1)\|=2$, $\|\nabla 0\|=0$, and $\|\nabla (\varphi_{\di}(x_1)\pm x_1)\|=|\varphi'_{\di}(x_1)\pm 1|$.  As $\di$ is perimeter preserving on convex bodies, $K(t,r)$ and $\di(K(t,r))$ have equal perimeters, so we obtain (see e.g.~\cite[p.~101]{EG})
\begin{eqnarray*}
\lefteqn{\int_{D(t,r)}\sqrt{1+2^2}\,dx+
\int_{D(t,r)}\sqrt{1+0^2}\,dx}\\
&=&\int_{D(t,r)}\sqrt{1+(\varphi'_{\di}(x_1)+ 1)^2}\,dx+\int_{D(t,r)}\sqrt{1+(\varphi'_{\di}(x_1)- 1)^2}\,dx.
\end{eqnarray*}
Dividing the previous equation by ${\mathcal{H}}^{n-1}(D(t,r))$ and taking the limit as $r\to 0$, Lebesgue's differentiation theorem (see e.g.~\cite[Theorem~8.8]{Rud87}) yields
$$\sqrt{5}+1=\sqrt{1+(\varphi'_{\di}(t)+ 1)^2}+\sqrt{1+(\varphi'_{\di}(t)- 1)^2}$$
for ${\mathcal{H}}^{1}$-almost all $t\in (0,\infty)$.  It is easy to check that the only solutions of the previous equation are $\varphi'_{\di}(t)=\pm 1$ for ${\mathcal{H}}^{1}$-almost all $t\in (0,\infty)$.

The above argument can be repeated with $t<0$, $0<r<-t$, and
$$K'(t,r)=\{x+te_n: x\in D(t,r),~~ -2x_1\le x_n\le 0\}.$$
This yields that $\varphi'_{\di}(t)=\pm 1$ for ${\mathcal{H}}^{1}$-almost all $t\in (-\infty,0)$, so the first statement in the theorem is proved.

Conversely, suppose that $\varphi$ is a contraction such that $|\varphi'(t)|=1$ for $t\in \R\setminus N$, where  ${\mathcal{H}}^{1}(N)=0$. Then (\ref{kseg}) defines a map $\di:\cK^n\rightarrow \cL^n$. Theorem~\ref{thmm7} implies that $\di$ is monotonic, measure preserving, and respects $H$-cylinders on $\cK^n_n$, but it is clear that these properties hold on $\cK^n$.

It remains to prove that $\di$ preserves perimeter on convex bodies.  Let $K\in \cK_n^n$.  We may assume without loss of generality that $\di:\cK^n\rightarrow \cL^n$ is defined by (\ref{kseg}) for all $x\in H$ and without the exceptional sets of $\cH^1$-measure zero.  Indeed, the difference, by Fubini's theorem, is a set of $\cH^n$-measure zero, which does not change perimeter (see \cite[Exercise~12.16]{Mag}).  Let
$$
f^+(x)=\max_{x+te_n\in K}t,\quad f^-(x)=\min_{x+te_n\in K}t,\quad g^+(x)=\max_{x+te_n\in \di K}t,\quad \text{ and }\quad g^-(x)=\min_{x+te_n\in \di K}t
$$
be the functions from $K|H$ to $\R$ whose graphs are the top and bottom parts of $\partial K$ and $\partial(\di K)$.   With the already established notation from Lemma~\ref{lemm11}, we have
\begin{equation}\label{mar101}
t_x=(f^+(x)+f^-(x))/2
\end{equation}
and
\begin{equation}\label{mar64}
g^\pm(x)=\varphi(t_x)\pm (f^+(x)-f^-(x))/2.
\end{equation}

Put $\Omega=\relint(K|H)$. For $E\subset\R^n$, let
$$E_1=(\partial E)\cap \left(((K|H)\setminus\Omega)+H^{\perp}\right)\quad{\text{and}}\quad E_2=(\partial E)\cap \left(\Omega+H^{\perp}\right).$$
Clearly $\partial K$ (or $\partial(\di K)$) is the disjoint union of $K_1$ and $K_2$ (or $(\di K)_1$ and $(\di K)_2$, respectively).  We show below that
\begin{equation}\label{mar131}
\cH^{n-1}((\di K)_i)=\cH^{n-1}(K_i)
\end{equation}
for $i=1,2$.  Assuming this is true, the proof is completed as follows.  Since $K$ is a convex body, it has Lipschitz boundary. The functions $\pm f^\pm$ are convex and hence locally Lipschitz on $\Omega$ (see \cite[Theorem~1.5.3]{Sch14}), so it follows from (\ref{mar101}) and (\ref{mar64}) that $g^\pm$ are also locally Lipschitz on $\Omega$.  From this and the fact that $\varphi$ is Lipschitz, it follows that $\di K$ also has Lipschitz boundary.  Note that $\partial K=\partial (\inte K)$, $\partial (\di K)=\partial(\inte (\di K))$, and $\cH^{n-1}(\partial(\di K))=\cH^{n-1}(\partial K)<\infty$ by (\ref{mar131}).  Then it follows from  \cite[Equation~(3.63)]{AFP} and \cite[Proposition~3.62)]{AFP}, applied to $\inte K$ and $\inte (\di K)$, and (\ref{mar131}), that
$$
S(\di K)=\cH^{n-1}(\partial (\di K))=\cH^{n-1}(\partial K)=S(K),$$
as required.

To prove (\ref{mar131}) for $i=1$, note that because $(K|H)\setminus \Omega$ is the boundary of a convex body in $H$, each $x\in (K|H)\setminus \Omega$ has a relative neighborhood $U\subset ((K|H)\setminus \Omega)$ such that there is a Lipschitz bijection from $U+H^{\perp}$ to $\R^{n-1}$, with Lipschitz inverse, that maps vertical lines isometrically to vertical lines.  Then the desired result follows directly from the area formula \cite[Theorem~1, p.~96]{EG}, the fact that
$$\cH^{1}((\di K)\cap (H^{\perp}+x))=\cH^{1}(K\cap (H^{\perp}+x))$$
for all $x\in (K|H)\setminus\Omega$, and Fubini's theorem in $\R^{n-1}$.

It therefore remains to prove that $\cH^{n-1}((\di K)_2)=\cH^{n-1}(K_2)$.  To this end, let $M=\{x\in \Omega: t_x\in N\}$. Applying the coarea formula \cite[p.~112]{EG} to the locally Lipschitz function $x\mapsto t_x$ gives
$$
\int_{M} |\nabla t_x| dx=\int_N \cH^{n-1}(\{x:t_x=s\})ds=0,
$$
as $\cH^1(N)=0$. Therefore $\cH^{n-1}(\{x\in M: \nabla t_x\ne o\})=0$.  As all the functions $f^{\pm}$ and $g^{\pm}$ are locally Lipschitz on $\Omega$, their gradients exist for $\cH^{n-1}$-almost all $x\in \Omega$.   Using (\ref{mar101}) and (\ref{mar64}), a direct calculation shows that
\begin{equation}\label{feb181a}
\nabla g^{\pm}(x)=\begin{cases}
\nabla f^{\pm}(x),& {\text{if $\varphi'(t_x)=1$ and $\nabla t_x\ne o$,}}\\
-\nabla f^{\mp}(x),& {\text{if $\varphi'(t_x)=-1$ and $\nabla t_x\ne o$.}}
\end{cases}
\end{equation}
Using the fact that $\varphi$ is a contraction, for $x\in \Omega$ with $\nabla t_x=0$ and $h\in H$ with $\|h\|$ sufficiently small, we have
$$\frac{|\varphi(t_{x+h})-\varphi(t_x)|}{\|h\|}\le   \frac{|t_{x+h}-t_x|}{\|h\|}=   \frac{|t_{x+h}-t_x-\langle \nabla t_x,h\rangle |}{\|h\|}\to 0
$$
as $h\to o$, implying that $\nabla \varphi(t_x)$ exists and is zero.  Then, taking gradients in (\ref{mar101}) and (\ref{mar64}), we obtain
\begin{equation}\label{mar102}
\nabla g^{\pm}(x)=\pm \nabla f^{+}(x)=\mp \nabla f^{-}(x), \quad {\text{if $\nabla t_x=o$,}}
\end{equation}
for $\cH^{n-1}$-almost all $x\in \Omega$. Using (\ref{feb181a}), (\ref{mar102}), and \cite[p.~101]{EG}, we get
\begin{eqnarray*}
\cH^{n-1}((\di K)_2)&=&\int_{\Omega}\left(\left(1+\|\nabla g^+(x)\|^2\right)^{1/2}+\left(1+\|\nabla g^-(x)\|^2\right)^{1/2}\right)\,dx\\
&=&\int_{\Omega}\left(\left(1+\|\nabla f^+(x)\|^2\right)^{1/2}+
\left(1+\|\nabla f^-(x)\|^2\right)^{1/2}\right)\,dx= \cH^{n-1}(K_2),
\end{eqnarray*}
as required.
\end{proof}

The equation $|\varphi'(t)|=1$, on a given domain and usually stated with boundary conditions, is a special case of the {\em eikonal equation}; see, for example, \cite[p.~47]{BCD}.  In addition to the four functions $\varphi(t)=t$, $-t$, $|t|$, and $-|t|$ in Corollary~\ref{corlemExtend}, there are infinitely many other solutions $\mathcal{H}^1$-almost everywhere on $\R$, including the function in the following example.

\begin{ex}\label{thm9exPolComp}
{\rm  Let $H=u^{\perp}$, $u\in S^{n-1}$, be oriented with $u\in H^+$. Define the contraction $\varphi(t)=\min_{k\in \Z}|t-k|$ for $t\in \R$ and let $\di:{\mathcal K}^n\rightarrow {\mathcal{L}}^n$ be defined by \eqref{kseg}.  Then $|\varphi'(t)|=1$ for ${\mathcal{H}}^{1}$-almost all $t\in \R$.  By Theorem~\ref{thmm9convex}, $\di$ is monotonic, measure preserving, respects $H$-cylinders, and is perimeter preserving on convex bodies.  By Lemma~\ref{ballmap}, $\di$ maps balls to balls and hence, by Corollary~\ref{cor6a}, it is invariant on $H$-symmetric sets.  However, $\di$ is not essentially equal to $\Id$, $\dagger$, $\di_{P_H}$, or $\di_{P_H}^{\dagger}$.
\qed}
\end{ex}

\begin{rem}\label{remthmm9convex}
{\rm  The functions appearing in Theorem~\ref{thmm9convex}, i.e., Lipschitz functions $\varphi: \R\to \R$ such that $|\varphi'(t)|=\pm 1$ almost everywhere, may be non-differentiable at an uncountable number of points.  In fact, given any set $N$ such that ${\cH}^1(N)=0$, there is a function $\varphi$ of this type such that $\varphi'$ does not exist at any point in $N$.  To see this, first note that by \cite[Theorem~1]{Goff}, there is a Borel set $E$ such that at each point in $N$, the upper Lebesgue density of $E$ is 1 and the lower Lebesgue density of $E$ is 0. (To construct $E$, let $(G_k)$ be a sequence of open sets such that $N\subset G_{k+1}\subset G_k$, ${\cH}^1(G_1)\le 1$, and for each component $C$ of $G_k$, ${\cH}^1(C\cap G_{k+1})\le (1/k){\cH}^1(C)$.  Then let $E=\cup_k(G_{2k-1}\setminus G_{2k})$.)  Now define
$$\varphi(t)=\int_0^t (2(1_E(s))-1)\,ds$$
for $t\in\R$. Then $\varphi$ is Lipschitz with $\varphi'(t)=1$ or $-1$ at each density point of $E$ or $\R\setminus E$, respectively, and hence almost everywhere by Lebesgue's density theorem, and it follows directly from the just-mentioned properties of $E$ that $\varphi'$ does not exist at any point in $N$.}
\end{rem}

\begin{lem}\label{lemmar9}
Let $H=u^{\perp}$, $u\in S^{n-1}$, and let ${\mathcal{K}}^n_n\subset{\mathcal{E}}\subset {\mathcal{L}}^n$, where ${\mathcal{E}}$ is closed under intersections with $H$-symmetric spherical cylinders. If $\di:{\mathcal{E}}\rightarrow {\mathcal{L}}^n$ is monotonic, measure preserving, and invariant on $H$-symmetric spherical cylinders, then $\di$ respects $H$-cylinders.
\end{lem}

\begin{proof}
In the proof, we ignore sets of $\cH^n$-measure zero.  If $r>0$ and $x\in H$, let $C(x,r)$ be the infinite spherical cylinder of radius $r$ and axis $H^{\perp}+x$.  Suppose that $E\in{\mathcal{E}}$ and $E\subset C(x,r)$.  Since $E\in\cL^n$, we can choose $s_m>0$, $m\in \N$, such that if $C_m=(C|H)+s_m[-u,u]$ and $E_m=E\cap C_m$, then $\cH^n(E\setminus E_m)\le 1/m$.  As $E_m\in {\mathcal{E}}$, $\di$ is monotonic, and $C_m$ is an $H$-symmetric spherical cylinder, we have $\di E_m\subset \di C_m=C_m$ and hence $F=\cup_m\di E_m\subset C(x,r)$.  Also, $\di E_m\subset \di E$ for $m\in \N$, implying that $F\subset \di E$. Now
$$\cH^n(F)\ge \cH^n(\di E_m)=\cH^n(E_m)\ge \cH^n(E)-1/m=\cH^n(\di E)-1/m$$
for each $m$, so $\cH^n(F)\ge \cH^n(\di E)$.  Thus $\di E=F\subset C(x,r)$, as required.
\end{proof}

All seven properties of maps $\di:{\mathcal{E}}\subset{\mathcal{L}}^n\to {\mathcal{L}}^n$ listed in Section~\ref{Properties} are shared by $\di_{P_H}$ and the map $\di$ of Example~\ref{thm9exPolComp}, so a further property is needed to distinguish $\di_{P_H}$.  When ${\mathcal{E}}$ contains $H$-symmetric unions of two disjoint balls, we can take inspiration from Example~\ref{Marex} and use the new property in the next result.  Note that the property is much weaker than invariance on all $H$-symmetric sets.

\begin{lem}\label{lemunion}
Let $H=u^{\perp}$, $u\in S^{n-1}$, and let ${\mathcal{K}}^n_n\subset{\mathcal{E}}\subset {\mathcal{L}}^n$, where ${\mathcal{E}}$ is closed under intersections with $H$-symmetric spherical cylinders.  If $\di:{\mathcal{E}}\rightarrow {\mathcal{L}}^n$ is monotonic, measure preserving, and defined and invariant on $H$-symmetric unions of two disjoint balls, then $\di$ is invariant on $H$-symmetric spherical cylinders and hence respects $H$-cylinders.
\end{lem}

\begin{proof}
Let $C=(B(x,r)\cap H)+s(B^n\cap H^{\perp})$, where $r,s>0$, be an $H$-symmetric spherical cylinder.  Let $A=\inte(C\setminus H)$.  Choose $H$-symmetric unions $U_m$ of two disjoint balls, $m\in \N$, such that $A=\cup_mU_m$.  By monotonicity and invariance on $H$-symmetric unions of two disjoint balls,
$$A=\cup_mU_m=\cup_m\di U_m\subset \di C.$$
Therefore, $C\subset \di C$, essentially.  The measure-preserving property of $\di$ now implies that $\di C=C$, essentially.  Thus $\di$ is invariant on $H$-symmetric spherical cylinders and hence respects $H$-cylinders by Lemma~\ref{lemmar9}.
\end{proof}

\begin{lem}\label{lemExtend}
Let $H=u^{\perp}$, $u\in S^{n-1}$, be oriented with $u\in H^+$ and let $\cE={\mathcal{C}}^n$ or ${\mathcal{L}}^n$. Suppose that $\di:\cE\rightarrow {\mathcal{L}}^n$ is monotonic, measure preserving, maps balls to balls, and invariant on $H$-symmetric unions of two disjoint balls.   Let $\varphi_{\di}:\R\to \R$ be the function from Lemma~\ref{lemm5}. Then $\varphi_{\di}(t)=t$, $-t$, $|t|$, or $-|t|$, if and only if $\di$ essentially equals $\mathrm{Id}$, $\dagger$, $\di_{P_H}$, or $\di_{P_H}^\dagger$, respectively.
\end{lem}

\begin{proof}
By Lemma~\ref{lemunion}, our assumptions on $\di$ imply that it also respects $H$-cylinders, allowing previous results to be applied.

If $\di$ is given, it is easy to check that $\varphi_{\di}$ has the appropriate form, by applying $\di$ to balls and using (\ref{ball}).

It will suffice show that if $\varphi_{\di}=|t|$, then $\di=\di_{P_H}$.  Indeed, the case when $\varphi_{\di}=-|t|$ then follows by applying the previous case to $\di^\dagger$. The cases when $\varphi_{\di}=\pm t$ are even simpler, not requiring the assumption that $\di$ is invariant on $H$-symmetric unions of two disjoint balls, but only that $\di$ respects $H$-cylinders.

In the rest of the proof, we shall ignore sets of ${\mathcal{H}}^n$-measure zero.

Assume initially that $E\in \cE$ is a countable union of balls, $E=\cup_{j=1}^\infty B_j$. Let $I^-$ be the set of all indices $j$ for which the center of $B_j$ lies in $H^-$, and put $I^+=\N\setminus I^-$.
By (\ref{ball}), $\di B_j=B_j$ for $j\in I^+$. Since $\di$ is monotonic, we obtain
\begin{equation}\label{eq1lem7}
\cup_{j\in I^+}B_j =\cup_{j\in I^+}\di B_j \subset \di E.
\end{equation}
For $j\in I^-$ we have  $\di B_j=B_j^\dagger$ by (\ref{ball}), so a similar argument gives
\begin{equation}\label{eq2lem7}
(\cup_{j\in I^-} B_j)^\dagger\subset \di E.
\end{equation}
We claim that
\begin{equation}\label{eq3lem7}
E\cap E^\dagger\subset\di E.
\end{equation}
To see this, note that
$$E\cap E^\dagger=\cup_{j,k=1}^\infty (B_j\cap B_k^\dagger)=\cup_{j\le k}^\infty E_{jk},$$
where $E_{jk}=(B_j\cap B_k^\dagger)\cup (B_k\cap B_j^\dagger)$ is an $H$-symmetric union of (at most) two compact convex sets.  It will therefore suffice to show that $\inte E_{jk}\subset\di E$ for $j,k\in \N$.  To this end, let $C_1,C_2,\ldots$ be balls such that
$$\inte (E_{jk}\cap H^+)=\cup_{i=1}^\infty C_i.$$
As $E_{jk}$ is $H$-symmetric, this implies that
$$C_i\cup C_i^\dagger\subset(\inte E_{jk})\setminus H\subset E\cap E^\dagger\subset E.$$
Hence,
$$\inte E_{jk}\subset
\cup_{i=1}^\infty (C_i\cup C_i^\dagger)=\cup_{i=1}^\infty \di(C_i\cup C_i^\dagger)\subset \di E,$$
where the equality is justified as $\di$ is invariant on $H$-symmetric unions of two disjoint balls, and the final containment follows by the monotonicity of $\di$.  This proves (\ref{eq3lem7}).

By (\ref{eq1lem7}), (\ref{eq2lem7}), and (\ref{eq3lem7}), if we can prove the first containment in
\begin{equation}
\label{eq:includiPh}
\di_{P_H} E\subset
\left(\cup_{j\in I_+} B_j\right)\cup\left(\cup_{j\in I_-}B_j\right)^\dagger\cup (E\cap E^\dagger) \subset \di E,
\end{equation}
the measure-preserving properties of $\di$ and $\di_{P_H}$ will show that $\di E=\di_{P_H}E$.  To prove the inclusion in question, fix $x\in (\di_{P_H}E)\setminus H^-$. Then $\{x,x^\dagger\}\cap E\ne\emptyset$ and there is a $j\in \N$ with $x\in B_j$ or  $x\in B_j^\dagger$.
Suppose that $x\in B_j$. If $j\in I^+$, the first inclusion in \eqref{eq:includiPh} holds trivially, and if $j\in I^-$ it holds as
$x\in B_j\setminus H^-\subset B_j^\dagger$. Similar arguments can be used if $x\in B_j^\dagger$. When
$x\in (\di_{P_H}E)\cap H^-$, we have $\{x,x^\dagger\}\subset E$, so
$x\in E\cap E^\dagger$.
Concluding, $\di E=\di_{P_H} E$ whenever $E\in \cE$ is a countable union of balls.  In particular, this is true if $E\in \cE$ is an open set.

Let $(E_m)$ be a decreasing sequence of sets in $\cE$ and let $E=\cap_m E_m\in \cE$.  We claim that if $\di E_m=\di_{P_H} E_m$ for $m\in \N$, then $\di E=\di_{P_H}E$.
We first show that
\begin{align}\label{eq:ContFromOutside}
\cap_m \left(\di_{P_H}E_m\right)=\di_{P_H}E.
\end{align}
To prove \eqref{eq:ContFromOutside}, observe that for $x\in H^-$ and $A\in {\mathcal{L}}^n$, we have $x\in \di_{P_H}A$ if and only if $\{x,x^\dagger\}\subset A$. Hence, when $x\in H^-$ we have
$$x\in \cap_m \left(\di_{P_H} E_m\right)
~~\Leftrightarrow~~\{x,x^\dagger\}\subset E_m,\text{ for }m\in\N
~~\Leftrightarrow~~ \{x,x^\dagger\}\subset E
~~\Leftrightarrow~~ x\in \di_{P_H}E.$$
For $x\in H^+$, we have $x\in \di_{P_H}A$ if and only if
$\{x,x^\dagger\}\cap A\ne \emptyset$.  Therefore
$$x\in \cap_m \left(\di_{P_H}E_m\right)
~~\Leftrightarrow~~\{x,x^\dagger\}\cap E_m\ne\emptyset, \text{ for }m\in\N
~~\Leftrightarrow~~\{x,x^\dagger\}\cap E\ne \emptyset
~~\Leftrightarrow~~x\in \di_{P_H}E,$$
where for the second equivalence, we used the fact that $(E_m)$ is decreasing, implying that if one of $x$ or $x^\dagger$ is not in $E_m$ for some $m$, then the other must be in $E_m$ for all $m$.  This proves \eqref{eq:ContFromOutside}.
From the monotonicity of $\di$, the assumption $\di E_m=\di_{P_H} E_m$ and  \eqref{eq:ContFromOutside}, we obtain
$$\di E\subset \cap_m \di E_m=\cap_m\di_{P_H}E_m=\di_{P_H}E,$$
and the measure-preserving properties of $\di$ and $\di_{P_H}$ show that $\di E=\di_{P_H}E$.  Concluding, the subclass of $\cE$ where $\di$ and $\di_{P_H}$ essentially coincide is closed under decreasing limits, if the intersection is contained in $\cE$.

Let $\cE={\cL}^n$.  If $E\in \cE$, there is a decreasing sequence $(E_m)$ of open sets whose intersection is essentially $E$.  By what was proved earlier, this concludes the proof for this case.

Now let $\cE={\mathcal{C}}^n$ and let $E\in \cE$. By compactness, there is a set $E_1$ that contains $E$ in its interior and is a finite union of balls with radius $1$. Using compactness again, we can find a finite union $E_2$ of balls with radius at most $1/2$ containing $C$ in its interior, such that $E_2\subset E_1$. Continuing this way, a decreasing sequence $(E_m)$ of finite unions of balls is constructed with $\cap_m E_m=E$. The first part of the proof shows that $\di E_m=\di_{P_H}E_m$, and the second part of the proof gives $\di E=\di_{P_H}E$.
\end{proof}

Let ${\mathcal{U}}^n$ be the family of countable unions of balls in $\R^n$.  Suppose that $\cE\subset {\mathcal{L}}^n$ is a class of measurable sets containing all balls and  such that for each $E\in \cE$, there is a decreasing sequence $(E_m)$ of sets in ${\mathcal{U}}^n$ whose intersection is essentially $E$.  Then the proof of the previous lemma shows that it holds for the class $\cE$.

\begin{thm}\label{thmm9measurable}
Let $H=u^{\perp}$, $u\in S^{n-1}$, be oriented with $u\in H^+$, let $\cE={\mathcal{C}}^n$ or ${\mathcal{L}}^n$, and suppose that $\di:{\mathcal{E}}\rightarrow {\mathcal{L}}^n$ is monotonic, measure preserving, perimeter preserving on convex bodies, and invariant on $H$-symmetric unions of two disjoint balls. Then $\di$ essentially equals $\Id$, $\dagger$, $\di_{P_H}$, or $\di_{P_H}^{\dagger}$.
\end{thm}

\begin{proof}
By Lemmas~\ref{ballmap} and~\ref{lemunion}, our assumptions imply that $\di$ respects $H$-cylinders and maps balls to balls, so Theorem~\ref{thmm7} implies that the restriction of $\di$ to $\cK^n$ is determined by (\ref{kseg}).  Let $t_0\neq 0$.  For $0<r<|t_0|$, the balls $B(\pm t_0u,r)$ are disjoint.
Since $\di$ is monotonic and invariant on $H$-symmetric unions of two disjoint balls, either $\di B(t_0u,r)=B(t_0u,r)$ or $\di B(t_0u,r)=B(-t_0u,r)$. It follows from (\ref{ball}) that $\varphi_\di(t_0)=\pm t_0$.  The continuity of $\varphi_\di$ implies that $\varphi_{\di}(t)=t$, $-t$, $|t|$, or $-|t|$, so the desired conclusion is provided by Lemma~\ref{lemExtend}.
\end{proof}

The map $\di:\cL^n\to\cL^n$ in Example~\ref{Marex} can also be considered as a map $\di: {\mathcal C}^n\to {\mathcal C}^n$.  We showed in Example~\ref{Marex} that $\di=\di_{P_H}$ on $\cK^n$.  This implies that $\di$ is perimeter preserving on convex bodies.  Consequently, $\di$ has all the properties assumed in the previous theorem except that it is not invariant on $H$-symmetric unions of two disjoint balls, showing that the latter property cannot be replaced by the weaker assumption of respecting $H$-cylinders.

The following examples deal with the other assumptions in Theorem~\ref{thmm9measurable}, where it is always assumed that $H=u^{\perp}$, $u\in S^{n-1}$.

\begin{ex}\label{thm9ex1}
{\rm  Given $E\in \cL^n$ with ${\mathcal{H}}^n(E)>0$, let
$$c_E=\frac{1}{{\mathcal{H}}^n(E)}\int_E x\,dx$$
be the center of gravity of $E$ and let $\di E$ be the reflection of $E$ in the hyperplane $H+c_E$.  Then $\di:\cE\to\cE$ for $\cE={\mathcal{C}}^n$ or $\cL^n$ is measure preserving, perimeter preserving on convex bodies, and invariant on $H$-symmetric unions of two disjoint balls.  (Indeed, $\di$ is clearly invariant on all $H$-symmetric sets.)  It is not monotonic, as can be seen by considering the double cone $\conv([-u,u]\cup (B^n\cap H))$ and its subset, the cone $\conv(\{u\}\cup (B^n\cap H))$.
\qed}
\end{ex}

\begin{ex}\label{thm9ex2}
{\rm  Let $\cE={\mathcal{C}}^n$ or $\cL^n$ and define $\di:\cE\to\cE$ by $\di E=\cl(\inte E)$.  Then $\di$ is monotonic, perimeter preserving on convex bodies, and invariant on $H$-symmetric unions of two disjoint balls, but not measure preserving.
\qed}
\end{ex}

\begin{ex}\label{thm9ex3}
{\rm For $x\in \R^n$, let
$$F(x)=\begin{cases}
x,& {\text{if $d(x,H)>1$,}}\\
x^{\dagger},& {\text{if $d(x,H)\le 1$}}
\end{cases}$$
and define $\di:\cE\to\cL^n$ for $\cE={\mathcal{C}}^n$ or $\cL^n$ by $\di E=F(E)$.  Then $\di$ is monotonic, measure preserving, and invariant on $H$-symmetric unions of two disjoint balls (indeed, on all $H$-symmetric sets), but does not preserve perimeter on convex bodies.
\qed}
\end{ex}

For maps $\di:{\mathcal K}^n\rightarrow {\mathcal{L}}^n$, invariance on $H$-symmetric unions of two disjoint balls is not available.  We therefore resort to a different and rather strong condition; we say that $\di$ is \emph{convexity preserving away from $H$} if $\di K$ is essentially convex (that is, $\di K$ coincides with a convex set up to a set of $\cH^n$-measure zero) for all $K\in {\mathcal K}^n$ with $K\cap H=\emptyset$.

\begin{thm}\label{thmmar14}
Let $H=u^{\perp}$, $u\in S^{n-1}$, be oriented with $u\in H^+$ and let $\di:{\mathcal K}_n^n\rightarrow {\mathcal{L}}^n$ be monotonic, measure preserving, invariant on $H$-symmetric sets, perimeter preserving on convex bodies, and convexity preserving away from $H$. Then $\di$ essentially equals $\Id$, $\dagger$, $\di_{P_H}$, or $\di_{P_H}^{\dagger}$.
\end{thm}

\begin{proof}
We may assume that $u=e_n$. Let $C=\{0\}\times [0,1]^{n-2}\times \{0\}$ be the unit cube in $H\cap e_1^\perp$.  For $r,t>0$, define
$$K(t,r)=\{x+x_ne_n: x\in C+[o,te_1],~~ r\le x_n\le r+2x_1\}.$$
Then $K(t,r)\in {\mathcal{K}}^n_n$ is disjoint from $H$ and if $x\in C+[o,te_1]$, then $x+(r+x_1)e_n$ is the midpoint of $K(t,r)\cap (\langle e_n\rangle+x)$.  From Theorem~\ref{thmm9convex}, we know that (\ref{kseg}) holds, where $|\varphi'_{\di}(t)|=1$ for $\cH^1$-almost all $t$.  Then
\begin{equation}\label{mar141}
\di K(t,r)=\{x+x_ne_n: x\in C+[o,te_1],~~ \varphi_{\di}(r+x_1)-x_1\le x_n\le \varphi_{\di}(r+x_1)+x_1\},
\end{equation}
essentially. Since $\di$ is convexity preserving away from $H$, there is a convex set $L=L(t,r)$ essentially equal to $\di K(t,r)$. As the boundary of a convex set has $\cH^n$-measure zero, we may assume that $L$ is closed. The sets $M_1=C+\varphi_\di(r)e_n$ and $M_2=C+[(\varphi_\di(r+t)-t)e_n, (\varphi_\di(r+t)+t)e_n]+te_1$ are contained in $L$, because any point in either set can be approximated by points in the right side of \eqref{mar141} that are also points in $L$. As $L$ is convex, $M=\conv(M_1\cup M_2)\subset L$. The wedge $M$ has the same volume as $K(t,r)$, so the measure-preserving property of $\di$ yields $\di K(t,r)=L=M$, up to sets of $\cH^n$-measure zero. Comparing the set in \eqref{mar141} with $M$ and using the continuity of $\varphi_\di$, we obtain
$$\varphi_\di(r+x_1)=\left(1-\frac{x_1}t\right)\varphi_\di(r)+
\frac{x_1}t\varphi_\di(r+t)
$$
for $0\le x_1\le t$. Letting $r\to 0+$, we conclude that $\varphi_\di$ is affine on $[0,t)$ for $t>0$.  As $\di$ is monotonic and invariant on $H$-symmetric sets, it respects $H$-cylinders, and our assumptions and Lemma~\ref{ballmap} ensure that it also maps balls to balls.  Therefore, by Corollary~\ref{cor6a}, $\varphi_\di(0)=0$ and so $\varphi_\di$ is linear.  As $|\varphi'_\di|=1$ almost everywhere, this implies that $\varphi_\di=\pm \mathrm{Id}$ on $[0,\infty)$. Similar arguments applied to
$$K'(t,r)=\{x+x_ne_n: x\in C+[te_1,o],~~ r-2x_1\le x_n\le r\},$$
where $t,r<0$, show that $\varphi_\di=\pm \mathrm{Id}$ on $(-\infty,0]$ and hence that $\varphi_{\di}(t)=t$, $-t$, $|t|$, or $-|t|$.   The proof is completed by Corollary~\ref{corlemExtend}.
\end{proof}

Example~\ref{thm9ex1} shows that the assumption that $\di$ is monotonic cannot be dropped in the previous theorem.  If $o\neq x_0\in H$, the map $\di K=K+x_0$ has all the properties in Theorem~\ref{thmmar14} except invariance on $H$-symmetric sets.  If we define $\di$ via (\ref{kseg}), where $\varphi_{\di}(t)=t/2$, then it is easily checked that $\di$ has all the properties except that it is not perimeter preserving.
Example~\ref{thm9exPolComp} shows that the assumption that $\di$ is convexity preserving away from $H$ cannot be omitted.

We do not have an example showing that the measure-preserving assumption is necessary, and this relates to an open problem stated as a variant of \cite[Problem~11.1]{BGG}, namely, does there exist a map from the convex bodies in $\R^n$ to those that are symmetric with respect to a fixed hyperplane $H$ that is monotonic, invariant on $H$-symmetric sets, and perimeter-preserving?  If such a map existed, it would provide the required example.

\begin{thm}\label{thmm10}
Let $H=u^{\perp}$, $u\in S^{n-1}$, be oriented with $u\in H^+$.  Let $X={\mathcal{S}}(\R^n)$ or $\cV(\R^n)$ and suppose that $T:X\rightarrow X$ is a rearrangement. If the induced map $\di_T$ defined by \eqref{eqIndic} is perimeter preserving on convex bodies and invariant on $H$-symmetric unions of two disjoint balls, then $T$ essentially equals $\Id$, $\dagger$, ${P_H}$, or ${P_H}^{\dagger}$.
\end{thm}

\begin{proof}
By Lemma~\ref{fromSetToFct} and Theorem~\ref{lemapril30}(i), $\di_T$ is well defined, monotonic, and measure preserving.  Together with our assumptions on $\di_T$, we can apply Theorem~\ref{thmm9measurable} with $\cE=\cL^n$ to conclude that $\di_T$ essentially equals $\Id$, $\dagger$, $\di_{P_H}$, or $\di_{P_H}^{\dagger}$. The proof is completed by Theorem~\ref{lemapril30}(ii).
\end{proof}

We now address the question of finding versions of the previous theorem for maps $T:X \to X$, where $X={\mathcal{M}}(\R^n)$ or ${\mathcal{M}}_+(\R^n)$.

\begin{thm}\label{may9cor}
Let $H=u^{\perp}$, $u\in S^{n-1}$, be oriented with $u\in H^+$.  Let $X={\mathcal{M}}(\R^n)$ (or $X={\mathcal{M}}_+(\R^n)$) and let $T:X\to X$ be a rearrangement. If the induced map $\di_T$ defined by \eqref{eqIndic} is perimeter preserving on convex bodies and invariant on $H$-symmetric unions of two disjoint balls, then the restriction of $T$ to ${\mathcal{S}}(\R^n)$ (or $\cV(\R^n)$, respectively) essentially equals $\Id$, $\dagger$, ${P_H}$, or ${P_H}^{\dagger}$.
\end{thm}

\begin{proof}
If $X={\mathcal{M}}(\R^n)$ (or $X={\mathcal{M}}_+(\R^n)$) and $T:X\to X$ is a rearrangement, then $T:{\mathcal{S}}(\R^n)\to {\mathcal{S}}(\R^n)$ (or $T:\cV(\R^n)\to \cV(\R^n)$, respectively).  This follows from Lemma~\ref{may8lem}(ii) (or the definitions of $\cV(\R^n)$ and equimeasurability, respectively). Since the restricted maps satisfy the hypotheses of Theorem~\ref{thmm10}, the result follows.
\end{proof}

Example~\ref{may1ex3} shows that the conclusion of Theorem~\ref{may9cor} cannot be drawn for the unrestricted maps $T:{\mathcal{M}}(\R^n)\to {\mathcal{M}}(\R^n)$ or $T:{\mathcal{M}}_+(\R^n)\to {\mathcal{M}}_+(\R^n)$.

\section{Appendix}

The purpose here is to compare our approach to rearrangements on ${\mathcal{S}}(\R^n)$ in Section~\ref{General functions} with that of \cite{BS} and \cite{VSW}.  Recall that we begin with a rearrangement $T:{\mathcal{S}}(\R^n)\to {\mathcal{S}}(\R^n)$ and show in Theorem~\ref{lemapril30}(ii) that $T$ is essentially determined by the associated measure-preserving and monotonic map $\di_T:\cL^n \to \cL^n$, defined by \eqref{eqIndic}, via the formula
\begin{equation}\label{eqoct201}
Tf(x)=\max\left\{\sup\{t\in \Q,\, t>\essinf f: x\in \di_T \{z: f(z)\ge t\}\},\essinf f\right\},
\end{equation}
which holds for $\cH^n$-almost all $x\in\R^n$, or alternatively via \eqref{eqnov13}.

Brock and Solynin \cite[Section~3]{BS} and Van Shaftingen and Willem \cite{VSW} reverse the procedure, starting with a set transformation $\di$ and defining function transformations. The latter paper, particularly, allows other possibilities, but we may focus on the case when $\di:\cL^n \to \cL^n$ is measure preserving and {\em pointwise monotonic} (meaning that if $K\subset L$, the containment $\di K\subset \di L$ must hold everywhere and not just almost everywhere).  Consider defining a function $T$ on ${\mathcal{S}}(\R^n)$ by
\begin{equation}\label{eqoct202}
Tf(x)=\max\{\sup\{t>\inf f: x\in \di \{z: f(z)> t\}\},\inf f\},
\end{equation}
for $x\in \R^n$.  The formula \eqref{eqoct202} is equivalent to \cite[Equation (3.1), p.~1762]{BS}, where $f$ is assumed continuous and it is shown that $Tf\in {\mathcal{S}}(\R^n)$.  For general $f\in {\mathcal{S}}(\R^n)$, Brock and Solynin suggest replacing the supremum in \eqref{eqoct202} by the essential supremum and then claim that $T:{\mathcal{S}}(\R^n)\to {\mathcal{S}}(\R^n)$ is a rearrangement.   In \cite{VSW}, the authors work with {\em admissible functions}, which in our context is equivalent to demanding that $\inf f=\essinf f$.  In \cite[Definition~4]{VSW}, the map $T$ defined by \eqref{eqoct202} is considered for admissible $f$ and denoted by $\breve{S}$, as well as a map $\overline{S}$ where in \eqref{eqoct202}, $\{z: f(z)> t\}$ is replaced by $\{z: f(z)\ge t\}$.  Among other results, it is stated in \cite[Propositions~1--3]{VSW} that under the above assumptions on $\di$, $T=\overline{S}=\breve{S}$ and the formula
\begin{equation}\label{eqoct211}
Tf(x)=\max\{\sup\{t>\essinf f: x\in \di \{z: f(z)\ge t\}\},\essinf f\},
\end{equation}
for $x\in \R^n$, defines a rearrangement $T:{\mathcal{S}}(\R^n)\to {\mathcal{S}}(\R^n)$.

Since the definition \eqref{eqoct211} appears to be a thoroughly measure-theoretic one, it is natural to ask whether the pointwise monotonicity of $\di$ is really required, or whether it could be replaced by monotonicity in our sense.  The following example shows that this is not the case.

\begin{ex}\label{exoct211}
{\em  Let $n=1$ and for $A\in \cL^1$, define
$$
\di A=\begin{cases}
A\cup\{\esssup A+1\},& {\text{if  $\esssup A<\infty$,}}\\
A,& {\text{otherwise.}}
\end{cases}
$$
Then $\di$ is essentially the identity on $\cL^1$ and thus measure preserving and monotonic, but it is not pointwise monotonic. For the function $f(x)=-|x|$ in ${\mathcal{S}}(\R^n)$ and $t\le 0$,
$$\di \{f\ge t\}=\di [t,-t]=[t,-t]\cup\{1-t\},$$
so by \eqref{eqoct211}, we have
\begin{eqnarray*}
Tf(x)&=&\sup\{t\le 0:x\in [t,-t]\cup\{1-t\} \}\\
&=&\sup\{t\le 0: t\le -|x| \text{ or } t=1-x\}=
\begin{cases}
1-x & {\text{if  $x\ge 1$,}} \\
-|x|,&  {\text{otherwise.}}
\end{cases}
\end{eqnarray*}
Hence, $T$ is not equimeasurable and therefore not a rearrangement.\qed
}
\end{ex}

One may wonder whether an alternative definition, namely,
\begin{equation}\label{eqoct212}
Tf(x)=\max\{\esssup\{t>\essinf f: x\in \di \{z: f(z)\ge t\}\},\essinf f\},
\end{equation}
for $x\in \R^n$, would allow the pointwise monotonicity assumption on $\di$ to be weakened, but the following example shows that it is consistent with ZFC that this is also not true.

\begin{ex}\label{exoct212}
{\em  Let $n=1$.  Assuming the continuum hypothesis CH, Sierpi\'nski \cite{Sie} constructed a set $S\subset [0,1]^2$ such that for $t\in [0,1]$, the horizontal section $S_t=\{x\in[0,1]: (x,t)\in S\}$ is countable and for $x\in [0,1]$, the vertical section
$S_x=\{t\in[0,1]: (x,t)\in S\}$ is such that $[0,1]\setminus S_x$ is countable.  For $A\in \cL^1$, define
$$
\di A=\begin{cases}
A\cup S_t,& {\text{if $A=[-t,0]$ for some $t\in [0,1]$,}}\\
A,& {\text{otherwise.}}
\end{cases}
$$
Then $\di$ is essentially the identity on $\cL^1$ and thus measure preserving and monotonic, but it is not pointwise monotonic. Put
$$
f(x)= \begin{cases}
1+x,& {\text{if  $x\le 0$,}}\\
-x, & {\text{otherwise.}}
\end{cases}
$$
Then $f\in {\mathcal{S}}(\R^n)$ and direct calculation using \eqref{eqoct212} shows that
$$
Tf(x)= \begin{cases}
1,& {\text{if  $x\in [0,1]$,}}\\
f(x), & {\text{otherwise,}}
\end{cases}
$$
so $T$ is not equimeasurable and hence not a rearrangement.
\qed}
\end{ex}

It would suffice in the previous example if the sections of $S$ satisfy $\cH^1(S_t)=0$ for all $t\in [0,1]$ and $\cH^1(S_x)=1$ for all $x\in [0,1]$. The existence of such an $S$ can be proved using Martin's Axiom MA and is therefore consistent with the negation of CH, while it is also consistent with ZFC that no such set exists.  See \cite[p.~673]{Lac} for these and other related remarks.

The supremum over $\Q$ in \eqref{eqoct201} cannot be replaced by the supremum over $\R$, i.e., the formula in Theorem~\ref{lemapril30}(ii) cannot be replaced by \eqref{eqoct211}.  Indeed, let $n=1$ and for $f\in {\mathcal{S}}(\R)$, define
$$
Tf=\begin{cases}
1_{A\cup\{\esssup A+1\}},& {\text{if $f=1_A$, where $A\in \cL^1$ and $\esssup A<\infty$,}}\\
f,& {\text{otherwise.}}
\end{cases}
$$
Then $T$ is essentially the identity and therefore a rearrangement on ${\mathcal{S}}(\R)$, so by Theorem~\ref{lemapril30}(ii), \eqref{eqoct201} holds. From its definition \eqref{eqIndic}, we see that $\di_T$ is the map $\di$ from Example~\ref{exoct211}, so it follows from that example that if $T$ satisfied \eqref{eqoct211}, it could not be equimeasurable, a contradiction.  In a similar way, using Example~\ref{exoct212} instead of Example~\ref{exoct211}, we see that it is consistent with ZFC that the supremum over $\Q$ in \eqref{eqoct201} cannot be replaced by the essential supremum over $\R$.

To the best of our knowledge, our result in Section~\ref{General functions} that every rearrangement $T:{\mathcal{S}}(\R^n)\to {\mathcal{S}}(\R^n)$ arises from the map $\di_T:\cL^n\to \cL^n$ defined by \eqref{eqIndic} has not been proved before.
The question of when a function transformation arises from a set transformation is addressed in \cite[Proposition~4]{VSW}.  This result appears to be based on \cite[Proposition~2.4.1]{VSPhD}, where we find the statement and proof clearer.   Restricting to our setting, it states that if an otherwise arbitrary $T:{\mathcal{S}}(\R^n)\to {\mathcal{S}}(\R^n)$ is such that
\begin{equation}\label{eqoct251}
\varphi\circ f\in {\mathcal{S}}(\R^n)~~\quad{\text{and}}~~\quad \varphi(Tf)=T(\varphi\circ f)
\end{equation}
for $f\in {\mathcal{S}}(\R^n)$, whenever $\varphi:\R\to\R$ is right-continuous and increasing, then there is a $\di:\cL^n\to \cL^n$ such that $T$ arises from $\di$ via the formula \eqref{eqoct211}. Note that $T$ need not be a rearrangement for \eqref{eqoct251} to hold.  For example, let $T:{\mathcal{S}}(\R^n)\to {\mathcal{S}}(\R^n)$ be the pointwise map defined by \eqref{polF} with $F(s,t)=\max\{s,t\}$, i.e.,  $Tf(x)=\max\{f(x), f(x^\dagger)\}$ for $x\in \R^n$.  Then $T$ is monotonic but not equimeasurable, while it is easy to check that it satisfies \eqref{eqoct251}.  (However, Lemma~\ref{lemnov1} implies that an equimeasurable map satisfying \eqref{eqoct251} must be monotonic and hence a rearrangement.)  On the other hand, \cite[Proposition~2.4.1]{VSPhD} and \cite[Proposition~4]{VSW} say nothing about rearrangements $T:{\mathcal{S}}(\R^n)\to {\mathcal{S}}(\R^n)$ until it is known that \eqref{eqoct251} is true.  This is just what we show holds, essentially, in Theorem~\ref{coroct241}, with a proof allowing the weakest possible continuity assumption on $\varphi$.


\begin{thebibliography}{999}

\bibitem{A}
L.~V.~Ahlfors, {\em Conformal Invariants: Topics in Geometric Function Theory}, McGraw-Hill Book Co., New York-D\"usseldorf-Johannesburg, 1973.

\bibitem{AFP}
L.~Ambrosio, N.~Fusco, and D.~Pallara, {\em Functions of Bounded Variation and Free Discontinuity Problems}, Oxford University Press, New York, 2000.

\bibitem{Baer}
A.~Baernstein II, {\em Symmetrization in Analysis}, with D.Drasin and R.~S.~Laugesen, Cambridge University Press, Cambridge, 2019.

\bibitem{BT}
A.~Baernstein II and B.~A.~Taylor, Spherical rearrangements, subharmonic functions, and *-functions in $n$-space, {\em Duke Math. J.} {\bf 43} (1976), 245--268.

\bibitem{B80}
C.~Bandle, {\em Isoperimetric Inequalities and Applications}, Pitman, Boston, MA, 1980.

\bibitem{BCD}
M.~Bardi and I.~Capuzzo-Dolcetta, {\em Optimal Control and Viscosity Solutions of Hamilton-Jacobi-Bellman equations}, Birkh\"{a}user, Boston, 1997.

\bibitem{BGG}
G.~Bianchi, R.~J.~Gardner, and P.~Gronchi, Symmetrization in geometry, {\em Adv. Math.} {\bf 306} (2017), 51--88.

\bibitem{BS}
F.~Brock and A.~Y.~Solynin, An approach to symmetrization via polarization, {\em Trans. Amer. Math. Soc.} {\bf 352} (2000), 1759--1796.

\bibitem{BurZ80}
Y.~D.~Burago and V.~A.~Zalgaller, {\em Geometric inequalities}, Springer, New York, 1980.

\bibitem{Bur}
A.~Burchard, Cases of equality in the Riesz rearrangement inequality,
{\em Ann. of Math.} (2) {\bf 143} (1996), 499--527.

\bibitem{BF}
A.~Burchard and M.~Fortier, Random polarizations, {\em Adv. Math.} {\bf 234} (2013), 550--573.

\bibitem{Burt}
G.~R.~Burton, Rearrangements of functions, saddle points and uncountable families of steady configurations for a vortex, {\em Acta Math.} {\bf 163} (1989), 291--309.

\bibitem{CR}
K.~M.~Chong and N.~M.~Rice, {\em Equimeasurable Rearrangements of Functions}, Queen's University, Kingston, Ont., 1971.

\bibitem{DeG}
E.~De Giorgi, Sulla propriet\`{a} isoperimetrica dell'ipersfera, nella classe degli insiemi aventi frontiera orientata di misura finita, {\em Atti Accad. Naz. Lincei. Mem. Cl. Sci. Fis. Mat. Nat. Sez. I} {\bf 5} (1958), 33--44.

\bibitem{DekV17}
J.~Dekeyser and J.~Van Schaftingen, Approximation of symmetrizations by Markov processes, {\em Indiana Univ. Math. J.} {\bf 66} (2017), 1145--1172.

\bibitem{Dou}
R.~J.~Douglas, Rearrangements of functions with applications to meteorology and ideal fluid flow, in: {\em Large-Scale Atmosphere-Ocean Dynamics, Vol. I}, pp.~288--341, Cambridge University Press, Cambridge, 2002.

\bibitem{Dub87}
V.~N.~Dubinin, Transformation of condensers in space, {\em Dokl. Akad. Nauk SSSR} {\bf 296} (1987), 18--20; translation in {\em Soviet Math. Dokl.} {\bf 36} (1988), 217--219.

\bibitem{Dub93}
V.~N.~Dubinin, Capacities and geometric transformations of subsets in $n$-space, {\em Geom. Funct. Anal.} {\bf 3} (1993), 342--369.

\bibitem{Dub14}
V.~N.~Dubinin, {\em Condenser Capacities and Symmetrization in Geometric Function Theory}, Springer, Basel, 2014.

\bibitem{EG}
L.~C.~Evans and R.~F.~Gariepy, {\em Measure Theory and Fine Properties of
Functions}, CRC Press, Boca Raton, FL, 1992.

\bibitem{F00}
L.~E.~Fraenkel, {\em An Introduction to Maximum Principles and Symmetry in Elliptic Problems}, Cambridge University Press, Cambridge, 2000.

\bibitem{Gar02}
R.~J.~Gardner, The Brunn-Minkowski inequality, {\em Bull. Amer. Math. Soc.}
{\bf 39} (2002), 355--405.

\bibitem{Gar06}
R.~J.~Gardner, {\em Geometric Tomography}, second edition, Cambridge University
Press, New York, 2006.

\bibitem{GK}
R.~J.~Gardner and M.~Kiderlen, Operations between functions, {\em Comm. Anal. Geom.} {\bf 26} (2018), 787--855.

\bibitem{Goff}
C.~Goffman, On Lebesgue's density theorem, {\em Proc. Amer. Math. Soc.} {\bf 1} (1950), 384--388.

\bibitem{Gru07}
P.~M.~Gruber, {\em Convex and Discrete Geometry}, Springer, Berlin, 2007.

\bibitem{HS}
C.~Haberl and F.~E.~Schuster, General $L_p$ affine isoperimetric inequalities, {\em J. Differential Geom.} {\bf 83} (2009), 1--26.

\bibitem{HacP}
T.~Hack and P.~Pivovarov, Randomized Urysohn-type inequalities, {\tt arXiv:1910.11654}.

\bibitem{Hay}
W.~K.~Hayman, {\em Multivalent Functions}, second edition, Cambridge University Press, Cambridge, 1994.

\bibitem{Hen}
A.~Henrot, {\em Extremum Problems for Eigenvalues of Elliptic Operators}, Birkh\"{a}user, Basel, 2006.

\bibitem{Kaw}
B.~Kawohl, {\em Rearrangements and Convexity of Level Sets in PDE},  Lecture Notes in Mathematics, 1150, Springer, Berlin, 1985.

\bibitem{K06}
S.~Kesavan, {\em Symmetrization \& Applications}, World Scientific Publishing, Hackensack, NJ, 2006.

\bibitem{KP}
S.~G.~Krantz and H.~R.~Parks, {\em The Geometry of Domains in Space}, Birkh\"{a}user, Boston, MA, 1999.

\bibitem{Lac}
M.~Laczkovich, Two constructions of Sierpi\'nski and some cardinal invariants of ideals, {\em Real Anal. Exchange} {\bf 24} (1998/99), 663--676.

\bibitem{LL}
E.~H.~Lieb and M.~Loss, {\em Analysis}, second edition, American Mathematical Society, Providence, RI, 2001.

\bibitem{LYZ1}
E.~Lutwak, D.~Yang, and G.~Zhang, $L_p$ affine isoperimetric inequalities, {\em J. Differential Geom.} {\bf 56} (2000), 111--132.

\bibitem{LYZ2}
E.~Lutwak, D.~Yang, and G.~Zhang, Orlicz centroid bodies, {\em J. Differential Geom.} {\bf 84} (2010), 365--387,

\bibitem{LZ}
E.~Lutwak and G.~Zhang, Blaschke-Santal\'o inequalities,
{\em J. Differential Geom.} {\bf 47} (1997), 1--16.

\bibitem{Mag}
F.~Maggi, {\em Sets of Finite Perimeter and Geometric Variational Problems.
An Introduction to Geometric Measure Theory}, Cambridge University Press, Cambridge, 2012.

\bibitem{PS}
G.~P\'{o}lya and G.~Szeg\"{o}, {\em Isoperimetric Inequalities in Mathematical Physics}, Princeton University Press, Princeton, NJ, 1951.

\bibitem{RS}
C.~A.~Rogers and G.~C.~Shephard, Some extremal problems for convex bodies,
{\em Mathematika} {\bf 5} (1958), 93--102.

\bibitem{Rud87}
W.~Rudin, {\em Real and Complex Analysis}, third edition, McGraw-Hill, New York, 1987.

\bibitem{Sch14}
R.~Schneider, {\em Convex Bodies: The Brunn-Minkowski Theory}, second edition, Cambridge University Press, Cambridge, 2014.

\bibitem{Sie}
W.~Sierpi\'nski, Sur un th\'eor\`eme \'equivalent \`a l'hypoth\`ese du continu, {\em Bull. Intern. Acad. Sci. Cracovie, A} 1919, 1--3.

\bibitem{Sol1}
A.~Y.~Solynin, Continuous symmetrization via polarization, {\em Algebra i Analiz} {\bf 24} (2012), 157--222; translation in {\em St. Petersburg Math. J.} {\bf 24} (2013), 117--166.

\bibitem{Str}
D.~W.~Stroock, {\em Essentials of Integration Theory for Analysis}, Springer, New York, 2011.

\bibitem{Tal}
G.~Talenti, The art of rearranging, {\em Milan J. Math.} {\bf 84} (2016), 105--157.

\bibitem{VSPhD}
J.~Van Schaftingen; R\'earrangements et probl\`emes elliptiques non lin\'eaires. Thesis (Ph.D.)-Universit\'e Catholique de Louvain. 2002. 94 pp.

\bibitem{VS1}
J.~Van Schaftingen, Universal approximation of symmetrizations by polarizations, {\em Proc. Amer. Math. Soc.} {\bf 134} (2005), 177--186.

\bibitem{VS2}
J.~Van Schaftingen, Explicit approximation of the symmetric rearrangement by polarizations, {\em Arch. Math.} {\bf 93} (2009), 181--190.

\bibitem{VSW}
J.~Van Schaftingen and M.~Willem, Set transformations, symmetrizations and isoperimetric inequalities, in: {\em Nonlinear analysis and applications to physical sciences}, pp.~135--152, Springer Italia, Milan, 2004.

\bibitem{Wol}
V.~Wolontis, Properties of conformal invariants, {\em Amer. J. Math.} {\bf 74} (1952), 587--606.

\end{thebibliography}
\end{document}